\documentclass[a4paper, twoside, 11pt, english]{article}

\usepackage[T1]{fontenc}
\usepackage[utf8]{inputenc}

\usepackage{etoolbox}
\usepackage{amsmath,amsthm,amssymb,stmaryrd}
\usepackage{mathtools}
\usepackage[ttscale=.875]{libertine}
\usepackage[libertine]{newtxmath}

\usepackage[numbers,sort&compress]{natbib}
\usepackage{comment}
\usepackage{ifthen}
\usepackage{algorithm2e}
\usepackage{multicol}

\usepackage{fancyhdr, titlesec, url, enumerate, microtype,setspace}

\usepackage[all,knot,poly]{xy}
\usepackage{tikz}
\usetikzlibrary{decorations.pathreplacing}

\usepackage{tikz-qtree,varwidth}
\usepackage{youngtab}
\usepackage{ytableau}

\usepackage[english]{babel}

\usepackage{hyperref}

\newcommand{\periodafter}[1]{\ifstrempty{#1}{}{#1.}}
\titleformat{\section}[block]{\scshape\filcenter\LARGE}{\thesection.}{.5em}{}
\titleformat{\subsection}[block]{\bfseries\filcenter\large}{\thesubsection.}{.5em}{\medskip}
\titleformat{\subsubsection}[runin]{\bfseries}{\thesubsubsection.}{.5em}{\periodafter}%{}[.]
\titlespacing{\subsubsection}{0pt}{\topsep}{.5em}

\titleformat{\section}[block]{\scshape\filcenter\Large}{\thesection.}{.5em}{}
\titleformat{\subsection}[runin]{\bfseries}{\thesubsection.}{.5em}{}[.]
\titleformat{\subsubsection}[runin]{\bfseries}{\thesubsubsection.}{.5em}{}[.]
\titlespacing{\subsubsection}{0pt}{10pt}{.5em}

\theoremstyle{ntheorem}
  	\newtheorem{theorem}[subsubsection]{Theorem}
  	\newtheorem{proposition}[subsubsection]{Proposition}
	\newtheorem{lemma}[subsubsection]{Lemma}
  	\newtheorem{corollary}[subsubsection]{Corollary}

\theoremstyle{definition}

	\makeatletter
\def\@equationname{equation}

\makeatother

\fancypagestyle{plain}{
\fancyhf{}\fancyfoot[LC,RC]{}
\fancyfoot[LE,RO]{$\thepage$}

}

\newcount\hh
\newcount\mm
\mm=\time
\hh=\time
\divide\hh by 60
\divide\mm by 60
\multiply\mm by 60
\mm=-\mm
\advance\mm by \time
\def\hhmm{\number\hh:\ifnum\mm<10{}0\fi\number\mm}

\UseTips
\SelectTips{eu}{11}

\newdir{ >}{{}*!/-10pt/@{>}}
\newdir{ -}{{}*!/-10pt/@{}}
\newdir{> }{{}*!/+10pt/@{>}}

\makeatletter

\xyletcsnamecsname@{dir4{}}{dir{}}
\xydefcsname@{dir4{-}}{\line@ \quadruple@\xydashh@}
\xydefcsname@{dir4{.}}{\point@ \quadruple@\xydashh@}
\xydefcsname@{dir4{~}}{\squiggle@ \quadruple@\xybsqlh@}
\xydefcsname@{dir4{>}}{\Tttip@}
\xydefcsname@{dir4{<}}{\reverseDirection@\Tttip@}

\xydef@\quadruple@#1{%
	\edef\Drop@@{%
		\dimen@=#1\relax
		\dimen@=.5\dimen@
		\A@=-\sinDirection\dimen@
		\B@=\cosDirection\dimen@
		\setboxz@h{%
			\setbox2=\hbox{\kern3\A@\raise3\B@\copy\z@}%
			\dp2=\z@ \ht2=\z@ \wd2=\z@ \box2
			\setbox2=\hbox{\kern\A@\raise\B@\copy\z@}%
			\dp2=\z@ \ht2=\z@ \wd2=\z@ \box2
			\setbox2=\hbox{\kern-\A@\raise-\B@\copy\z@}%
			\dp2=\z@ \ht2=\z@ \wd2=\z@ \box2
			\setbox2=\hbox{\kern-3\A@\raise-3\B@ \noexpand\boxz@}%
			\dp2=\z@ \ht2=\z@ \wd2=\z@ \box2
		}%
		\ht\z@=\z@ \dp\z@=\z@ \wd\z@=\z@ \noexpand\styledboxz@
	}%
}

\xydef@\Tttip@{\kern2pt \vrule height2pt depth2pt width\z@
	\Tttip@@ \kern2pt \egroup
	\U@c=0pt \D@c=0pt \L@c=0pt \R@c=0pt \Edge@c={\circleEdge}%
	\def\Leftness@{.5}\def\Upness@{.5}%
	\def\Drop@@{\styledboxz@}\def\Connect@@{\straight@{\dottedSpread@\jot}}}
	
\xydef@\Tttip@@{%
	\dimen@=.25\dimen@
 	\B@=\cosDirection\dimen@
	\setboxz@h\bgroup\reverseDirection@\line@ \wdz@=\z@ \ht\z@=\z@ \dp\z@=\z@
	{\vDirection@(1,-1)\xydashl@ \xyatipfont\char\DirectionChar}%
	{\vDirection@(1,+1)\xydashl@ \xybtipfont\char\DirectionChar}%
}

\xydef@\ar@form{
	\ifx \space@\next \expandafter\DN@\space{\xyFN@\ar@form}%
	\else\ifx ^\next \DN@ ^{\xyFN@\ar@style}\edef\arvariant@@{\string^}%
	\else\ifx _\next \DN@ _{\xyFN@\ar@style}\edef\arvariant@@{\string_}%
	\else\ifx 0\next \DN@ 0{\xyFN@\ar@style}\def\arvariant@@{0}%
	\else\ifx 1\next \DN@ 1{\xyFN@\ar@style}\def\arvariant@@{1}%
	\else\ifx 2\next \DN@ 2{\xyFN@\ar@style}\def\arvariant@@{2}%
	\else\ifx 3\next \DN@ 3{\xyFN@\ar@style}\def\arvariant@@{3}%
	\else\ifx 4\next \DN@ 4{\xyFN@\ar@style}\def\arvariant@@{4}%
	\else\ifx \bgroup\next \let\next@=\ar@style
	\else\ifx [\next \DN@[##1]{\ar@modifiers{[##1]}}%]
	\else\ifx *\next \DN@ *{\ar@modifiers}%
	\else\addLT@\ifx\next \let\next@=\ar@slide
	\else\ifx /\next \let\next@=\ar@curveslash
	\else\ifx (\next \let\next@=\ar@curveinout %)
	\else\addRQ@\ifx\next \addRQ@\DN@{\ar@curve@}%
	\else\addLQ@\ifx\next \addLQ@\DN@{\xyFN@\ar@curve}%
	\else\addDASH@\ifx\next \addDASH@\DN@{\defarstem@-\xyFN@\ar@}%
	\else\addEQ@\ifx\next \addEQ@\DN@{\def\arvariant@@{2}\defarstem@-\xyFN@\ar@}%
	\else\addDOT@\ifx\next \addDOT@\DN@{\defarstem@.\xyFN@\ar@}%
	\else\ifx :\next \DN@:{\def\arvariant@@{2}\defarstem@.\xyFN@\ar@}%
	\else\ifx ~\next \DN@~{\defarstem@~\xyFN@\ar@}%
	\else\ifx !\next \DN@!{\dasharstem@\xyFN@\ar@}%
	\else\ifx ?\next \DN@?{\ar@upsidedown\xyFN@\ar@}%
	\else \let\next@=\ar@error
	\fi\fi\fi\fi\fi\fi\fi\fi\fi\fi\fi\fi\fi\fi\fi\fi\fi\fi\fi\fi\fi\fi\fi \next@}

\makeatother

\newcommand{\fl}{\rightarrow}

\newcommand{\qfl}{\xymatrix@1@C=10pt{\ar@4 [r] &}}

\renewcommand{\tilde}[1]{\widetilde{#1}}

\DeclareMathOperator{\U}{\mathcal{U}}

\DeclareMathOperator{\Si}{\mathcal{S}}
\DeclareMathOperator{\Grw}{Grw}
\DeclareMathOperator{\Rec}{Rect}

\DeclareMathOperator{\YoungT}{Yt}
\DeclareMathOperator{\SkewT}{St}

\renewcommand{\phi}{\varphi}
\renewcommand{\epsilon}{\varepsilon}

\newcommand{\Zb}{\mathbb{Z}}

\renewcommand{\Pr}{\EuScript{P}}

\newcommand{\Yr}{\EuScript{Y}}

\newcommand{\M}{\mathbf{M}}

\renewcommand{\leq}{\leqslant}
\renewcommand{\geq}{\geqslant}

\DeclareMathOperator{\rem}{Rem}

\newcommand{\insl}[1]{\rightsquigarrow_{#1}}
\newcommand{\insr}[1]{\;\raisebox{0.1em}{{\rotatebox[origin=c]{180}{$\rightsquigarrow$}}}_{#1}\;}

\definecolor{cyan}{RGB}{175,238,238} 
\definecolor{Red}{rgb}{0.96,0.17,0.20}
\definecolor{RedD}{rgb}{0.57, 0.0, 0.04}
\definecolor{Green}{rgb}{0.0,1.0,0.0}
\definecolor{GreenL}{rgb}{0.0,1.0,0.0} 
\definecolor{GreenD}{rgb}{0.64,0.76,0.68} 
\definecolor{Yellow}{rgb}{1.0,1.0,0.19}
\definecolor{YellowD}{rgb}{0.93, 0.84, 0.25}
\definecolor{Blue}{rgb}{0.0,1.0,1.0}
\definecolor{BlueD}{rgb}{0.63,0.79,0.95}
\definecolor{vert}{rgb}{0,0.45,0}
\definecolor{bazaar}{rgb}{0.6, 0.47, 0.48}
\definecolor{bronze}{rgb}{0.8, 0.5, 0.2}
\definecolor{darkspringgreen}{rgb}{0.09, 0.45, 0.27}

\def\Pr{\mathcal{P}}

\def\Yr{\mathcal{Y}}
\makeatletter
\def\blfootnote{\xdef\@thefnmark{}\@footnotetext}
\makeatother

\def\P{\mathbf{P}}
\def\T{\mathbf{T}}

\definecolor{Red}{rgb}{0.96,0.17,0.20}
\definecolor{RedD}{rgb}{0.57, 0.0, 0.04}
\definecolor{Green}{rgb}{0.0,1.0,0.0}
\definecolor{GreenL}{rgb}{0.0,1.0,0.0} 
\definecolor{GreenD}{rgb}{0.64,0.76,0.68} 
\definecolor{Yellow}{rgb}{1.0,1.0,0.19}
\definecolor{YellowD}{rgb}{0.93, 0.84, 0.25}
\definecolor{Blue}{rgb}{0.0,1.0,1.0}
\definecolor{BlueD}{rgb}{0.63,0.79,0.95}
\definecolor{vert}{rgb}{0,0.45,0}
\definecolor{bazaar}{rgb}{0.6, 0.47, 0.48}
\definecolor{bronze}{rgb}{0.8, 0.5, 0.2}
\definecolor{darkspringgreen}{rgb}{0.09, 0.45, 0.27}

\begin{document}
\thispagestyle{empty}

\begin{center}

\begin{doublespace}
\begin{huge}
{\scshape  Super jeu de taquin and combinatorics }

{\scshape  of super tableaux of type A}
\end{huge}

\bigskip
\hrule height 1.5pt 
\bigskip

\begin{Large}
{\scshape Nohra Hage }
\end{Large}
\end{doublespace}

\bigskip

\begin{small}\begin{minipage}{14cm}
\noindent\textbf{Abstract -- } 
This paper presents a combinatorial study of the super plactic monoid of type A, which is related to the representations of the general linear Lie superalgebra.  We introduce the analogue of the Sch\"{u}tzenberger's jeu de taquin on the structure of super tableaux over a signed alphabet. We show that this procedure which transforms super skew tableaux into super Young tableaux is compatible with the super plactic congruence and it is confluent. We deduce properties relating the super jeu de taquin to insertion algorithms on super tableaux. Moreover, we introduce  the super evacuation procedure as an involution on super tableaux and we show its compatibility with the super plactic congruence.  
Finally, we describe the super jeu de taquin in terms of Fomin's growth diagrams in order to give a combinatorial version of the super Littlewood--Richardson rule. 

\medskip

\smallskip\noindent\textbf{Keywords --}  super jeu de taquin, super Young tableaux,  super plactic monoids,   super Littlewood--Richardson rule.

\medskip

\smallskip\noindent\textbf{M.S.C. 2010 -- Primary:} 05E99. \textbf{Secondary:} 20M05, 05A99.
\end{minipage}\end{small}
%\end{center}

\tikzset{every tree node/.style={minimum width=1em,draw,circle},
         blank/.style={draw=none},
         edge from parent/.style=
         {draw,edge from parent path={(\tikzparentnode) -- (\tikzchildnode)}},
         level distance=0.8cm}

\vspace{1cm}

\begin{small}\begin{minipage}{12cm}
\renewcommand{\contentsname}{}
\setcounter{tocdepth}{1}
\tableofcontents
\end{minipage}
\end{small}
\end{center}

\tikzset{every tree node/.style={minimum width=1em,draw,circle},
         blank/.style={draw=none},
         edge from parent/.style=
         {draw,edge from parent path={(\tikzparentnode) -- (\tikzchildnode)}},
         level distance=0.8cm}
 \bigskip

\section{Introduction}
         
The structure of plactic monoids has its origins in the works of Schensted~\cite{Schensted61} and Knuth~\cite{Knuth70} in the combinatorial study of Young tableaux~\cite{Young28}, and it has became a ubiquitous tool in algebraic combinatorics, representation theory and probabilistic combinatorics,~\cite{Lothaire02,Fulton97,OConnell03}. Schensted introduced an insertion algorithm on Young tableaux in order to compute the length of  the longest decreasing and increasing subsequences of a given word over the totally ordered alphabet~\mbox{$[n]:=\{1<\ldots<n\}$.} Knuth showed that two words over~$[n]$ have the same Young tableau after applying Schensted's insertion if, and only, if they are related by a sequence of the following \emph{Knuth relations}:
 \begin{equation}
\label{E:KnuthRelations}
zxy = xzy \; \text{ for }\; 1\leq x\leq y < z \leq n\quad \text{ and }\quad   yzx = yxz \; \text{ for }\; 1\leq x<y\leq z\leq n.
\end{equation}
Lascoux and Sch\"{u}tzenberger introduced in~\cite{LascouxSchutsenberger81} the plactic monoid as the quotient of the free monoid over~$[n]$ by the congruence generated by the family of Knuth relations~\eqref{E:KnuthRelations}, and they investigated its algebraic and combinatorial properties. 
Using Kashiwara's theory of crystal bases,  the plactic  monoid is also related to the representations of the special linear Lie algebra~$\mathfrak{s}\mathfrak{l}_n$ of type~$A_{n+1}$ and it is known as the plactic monoid of type~A,~\cite{LascouxLeclercThibon95}. In this way, the \emph{plactic monoid} of rank~$n$, denoted by~$\P([n])$, is the quotient of the free monoid over~$[n]$ by the congruence relation~$\sim_{\P([n])}$ defined in the following equivalent three ways:
\begin{enumerate}[\bf D1)] 
\item \emph{Insertion and tableaux}:~$\sim_{\P([n])}$ relates those words that yield the same Young tableau as the result of Schensted's insertion algorithm.
\item \emph{Defining relations}:~$\sim_{\P([n])}$ is defined to be the congruence generated by the family of Knuth relations~\eqref{E:KnuthRelations}.
\item \emph{Crystal}:~$\sim_{\P([n])}$ is defined by~$u\sim_{\P([n])} v$ if and only if there is a \emph{crystal isomorphism} between connected components of the crystal graph of the vector representation of~$\mathfrak{s}\mathfrak{l}_n$, that maps~$u$ to~$v$.
\end{enumerate}
Plactic monoids are also defined for other finite-dimensional semisimple Lie algebras of classical types and exceptional ones using Kashiwara's theory of crystal bases,~\cite{Lecouvey07}, or Littelmann's  path model,~\cite{Littelmann96}.
Sch\"{u}tzenberger introduced in~\cite{Schutzenberger77} the \emph{jeu de taquin} procedure on the structure of tableaux in order to give one of the first correct proofs of the \emph{Littlewood--Richardson rule} using properties of the plactic monoid. This rule describes in a combinatorial way the multiplicity of a \emph{Schur polynomial} in a product of Schur polynomials, that is, the multiplicity of an irreducible representation of the general Lie algebra
in a tensor product of two irreducible representations.
The jeu de taquin has found many applications in algebraic combinatorics and probabilistic combinatorics,~\cite{Fulton97,Leeuwen01, RomikSniady15}.
Recently, plactic monoids and other similar monoids are studied by rewriting methods,~\cite{CainGrayMalheiro14, CainGrayMalheiro15b, Hage15, HageMalbos17, HageMalbos22}.
Note finally that structural properties for plactic algebras are obtained in~\cite{CedoOkninski04}, and tropical representations for finite-rank plactic monoids are constructed in~\cite{JohnsonKambites21}, which implies that every plactic monoid of finite rank satisfies a non-trivial identity.

In this paper, we study the super version of the plactic monoid of type~A over a signed alphabet. A \emph{signed alphabet} is a finite or countable totally ordered set~$\Si$ which is the disjoint union of two subsets~$\Si_0$ and~$\Si_1$. 
Suppose~$\Si_0<\Si_1$, that is every element of~$\Si_0$ is less than every element of~$\Si_1$. In this case, the \emph{super plactic monoid} over a signed alphabet~$\Si$, denoted by~$\P(\Si)$, is the quotient of the free monoid~$\Si^\ast$ over~$\Si$ by the congruence relation~$\sim_{\P(\Si)}$ defined in the following equivalent three ways:
\begin{enumerate}[\bf D1)] 
\item \emph{Insertion and tableaux}:~$\sim_{\P(\Si)}$ relates those words that yield the same super  tableau as the result of the Schensted-like insertion algorithm introduced in~\cite{LaScalaNardozzaSenato06}.
\item \emph{Defining relations}:~$\sim_{\P(\Si)}$  is defined to be the congruence generated by the following family of  \emph{super  Knuth-like relations},~\cite{BenkartKangKashiwara20, LaScalaNardozzaSenato06}:
\begin{equation}
\label{Eq:SuperKnuthRelations}
\begin{array}{rl}
xzy=zxy,\;\text{ with }\; x=y \;\text{ only if }\; y\in\Si_0 \;\text{ and } \; y=z\; \text{ only if }\; y \in \Si_1,\\
yxz=yzx,\;\text{ with }\; x=y \;\text{ only if }\;  y \in \Si_1\;\text{ and } \; y=z\; \text{ only if }\; y \in \Si_0,
\end{array}
\end{equation}
for any~$x\leq y\leq z$ of elements of~$\Si$. 
\item \emph{Crystal}:~$\sim_{\P(\Si)}$ is defined by~$u\sim_{\P(\Si)} v$ if and only if there is a crystal isomorphism between connected components of the crystal graph of the vector representation of the general linear Lie superalgebra~$\mathfrak{g}\mathfrak{l}_{m,n}$, that maps~$u$ to~$v$,~\cite{BenkartKangKashiwara20}.
\end{enumerate}
It is worth noting that the relations~\eqref{Eq:SuperKnuthRelations} are the reverse of the ones presented in~\cite{BenkartKangKashiwara20}, because the reading maps in~\cite{BenkartKangKashiwara20} read super tableaux column-wise or row-rise from right to left and from top to bottom whereas the reading maps used in this paper read super tableaux in the reverse direction. We will follow the definition of the super plactic monoid introduced in~\cite{LaScalaNardozzaSenato06} that does not take into account the condition~$\Si_0<\Si_1$, and we will show all the results for any
ordering of~$\Si$. 
Note that when $\Si=\Si_0=[n]$, we recover the definition of the plactic monoid of type~A. 
Note also that the super plactic monoid appeared in~\cite{LodayPovov08} as a deformation of the parastatistics algebra which is is a superalgebra with even parafermi and odd parabose creation and annihilation operators. 
Moreover, super algebraic structures have found many applications as combinatorial tools in the study of  the invariant theory of superalgebras, the representation theory of general Lie super algebras, and algebras satisfying identities,~\cite{BereleRemmel85, BereleRegev87, BonettiSenatoVenezia98,GrosshansRotaStein87}.

Let~$\Si$ be a signed alphabet and let~$\lambda=(\lambda_1,\ldots,\lambda_k)\in\mathbb{N}^k$ be a \emph{partition} of a positive integer~$n$, that is a weakly decreasing sequence such that~$ \sum\lambda_i = n$.
A \emph{super tableau} of \emph{shape}~$\lambda$ over~$\Si$  is a collection of boxes in a left-justified row where the $i$-th row contains $\lambda_i$ boxes, for $1\leq i\leq k$, and filled by elements of~$\Si$ such that the entries in each row are weakly increasing from left to right allowing the repetition only of elements in~$\Si_0$ and the entries in each column are weakly increasing from top to bottom allowing the repetition only  of elements in~$\Si_1$.  We will denote by~$\YoungT(\Si)$ the set of all super tableaux over~$\Si$, and by~$R_{row}:\YoungT(\Si)\to \Si^\ast$
the map that reads the entries of a super tableau row-wise from bottom to top and from left to right. 
Consider two partitions~$\lambda$ and~$\mu$ such that~$\mu_i\leq \lambda_i$, for any~$i$.
A \emph{super skew tableau} of \emph{shape}~$\lambda/\mu$  over~$\Si$ is  obtained by removing  the super tableau of shape~$\mu$ from the super  tableau of shape~$\lambda$, and filled by elements of~$\Si$ such that the entries in each row are weakly increasing from left to right  with respect to~$\Si_0$ and the entries in each column are weakly increasing from top to bottom with respect to~$\Si_1$.
Schensted-like left and right insertion algorithms, denoted respectively by~$\insl{}$ and~$\insr{}$, are  introduced in~\cite{LaScalaNardozzaSenato06}, and consist in inserting elements of~$\Si$ into super  tableaux  by rows and columns respectively.  
The \emph{cross-section property} of super tableaux with respect to the \emph{super plactic congruence} generated by the relations~\eqref{Eq:SuperKnuthRelations},  is proved in~\cite{LaScalaNardozzaSenato06}, namely two words over~$\Si$ are super plactic congruent if and only if they lead to the same super tableau after applying the super analogue of Schensted's insertion algorithm.
As a consequence,  we deduce that the internal product~$\star_{\YoungT(\Si)}$ defined on $\YoungT(\Si)$ by setting~$t \star_{\YoungT(\Si)} t' :=  (t\insr{} R_{row}(t'))$
for all~$t$ and~$t'$ in $\YoungT(\Si)$, is associative.
Hence, the set~$\YoungT(\Si)$ equipped with the product~$\star_{\YoungT(\Si)}$ forms a monoid that is isomorphic to  the super plactic monoid.

We introduce in Section~\ref{S:SuperJeudeTaquin} the \emph{super jeu taquin} as an algorithm on the structure of super tableaux consisting in applying  successively \emph{forward sliding operations} on a  super skew tableau that move an inner corner into an outer position by keeping the rows and the columns weakly increasing until no more inner corners remain in the initial super skew tableau.  We prove 
 that the \emph{rectification} of a super skew tableau~$S$ by the super jeu de taquin is the unique super  tableau whose reading is equivalent to the reading of~$S$ with respect to the super plactic congruence.
As a consequence, we show in Theorem~\ref{T:ConfluenceJeuDeTaquin} that the super jeu de taquin is \emph{confluent}, that is, the resulting super  tableau does not depend on the order in which we choose inner corners in the forward sliding procedure.  Explicitly, if there are two sequences of sliding operations that transform a super skew tableau~$S$ into two different super tableaux~$S_1$ and~$S_2$, then we continue applying sliding operations until we reach two super  tableaux~$\tilde S_1$  and~$\tilde S_2$ without inner corners.
Since the readings of~$\tilde S_1$ and $\tilde S_2$  are super plactic congruent, we deduce that $\tilde S_1=\tilde S_2$,  by the cross-section property. Moreover, we relate in Subsection~\ref{SS:InsertionsSuperJeuDeTaquin} the super jeu de taquin to the insertion algorithms and we show how we can insert a  super tableau into another one by taquin. 

Using the super jeu de taquin, we introduce in Section~\ref{S:SuperJeuDeTaquinSuperEvacuation} the super analogue of the Sch\"{u}tzenberger's  evacuation procedure which transforms a super tableau~$t$ over a signed alphabet~$\Si$ into an \emph{opposite tableau}, denoted by~$t^{\mathrm{op}}$, over the \emph{opposite alphabet}~$\Si^{op}$  obtained from~$\Si$ by reversing its order. We show in Theorem~\ref{T:DualityTableauTheorem} that the super tableaux~$t$ and~$t^{\mathrm{op}}$ have the same shape and that the map~$t\mapsto  t^{\mathrm{op}}$ is an involution on~$\YoungT(\Si)$ that is compatible with the super plactic congruence.

Following Fomin's construction of growth diagrams~\cite{Fomin99}, we give in Section~\ref{S:SuperLittlewoodRichardsonRule} an interpretation of the super jeu de taquin in terms of \emph{super growth diagrams}.  Using this diagrammatic interpretation, we deduce that the super jeu de taquin  is \emph{symmetric}, that is, the corresponding super growth diagram can be constructed either starting from its topmost row and leftmost column or starting from its rightmost column and  bottom-most row using \emph{Fomin's local rule}. As a consequence, we show in Theorem~\ref{T:SuperLittlewoodRichardsonRule}  that the number of super skew tableaux of shape~$\lambda/\mu$  that rectify to a given super tableau~$t$ of shape~$\nu$, denoted by~$c_{\lambda,\mu}^{\nu}$,  does not depend on~$t$, and depends only on the partitions~$\lambda$,~$\mu$  and~$\nu$. Moreover, we show  that  the following identity
\[
S_{\lambda/\mu}\;=\; \underset{\nu}{\sum} c_{\lambda,\mu}^{\nu} S_{\nu}
\]
holds in the tableau~$\mathbb{Z}$-algebra rising from the super plactic monoid, where~$S_{\lambda/ \mu}$ and~$S_{\nu}$ denote respectively the formal sum of all super tableaux of shape~$\lambda/\mu$ and~$\nu$  over a finite signed alphabet.

\subsubsection*{Notation}
Let~$\mathcal{A}$ be  a totally ordered alphabet. We will denote by $\mathcal{A}^\ast$ the free monoid of \emph{words over~$\mathcal{A}$}, the product being concatenation of words, and the identity being the empty word. When~$\mathcal{A}$ is finite, we will denote by~$\# \mathcal{A}$ its cardinal number. We will denote by $w=x_1\ldots x_k$ a word in~$\mathcal{A}^\ast$ of \emph{length} $k$, where $x_1,\ldots,x_k$ belong to~$\mathcal{A}$. The length of a word~$w$ will be denoted by~$|w|$. Let~$w=x_1\ldots x_k$ be a word in~$\mathcal{A}^\ast$. 
The \emph{mirror word} of~$w$  is the word~$x_k\ldots x_1$ obtained by reversing its elements.
We denote by~$\ell(w)$ the leftmost letter of $w$ and by~$\rem(w)$ the subword of $w$ such that~$w=\ell(w)\rem(w)$. A word~$w'$ is a \emph{subsequence} of~$w$ if~$w'=x_{i_1}\ldots x_{i_l}$ with~$1\leq i_1< \ldots <i_l<k$. We will denote by~$[n]$ the ordered set~$\{1<2<\ldots<n\}$ for $n$ in $\Zb_{>0}$.
Let $\Si$ be a finite or countable totally ordered  set and~$||.||:\Si\to \mathbb{Z}_2$ be any map, where $\mathbb{Z}_2=\{0,1\}$ denotes the additive cyclic group of order $2$. The ordered pair $(\Si,||.||)$ is called  a \emph{signed alphabet} and we denote $\Si_0 =\{a\in \Si \;\big|\; ||a||=0\}$ and $\Si_1 =\{a\in \Si \;\big|\; ||a||=1\}$.  
A monoid~$\M$ is said a \emph{$\Zb_2$-graded monoid} or a \emph{supermonoid} if a map $||.||:\M\to \Zb_2$ is given such that $||u.v||= ||u||+||v||$, for all~$u$ and~$v$ in~$\M$. We call $||u||$ the $\Zb_2$-degree of the element~$u$. The free monoid~$\Si^\ast$ over~$\Si$ is  $\Zb_2$-graded  by considering~$||w||:=||x_1||+\ldots+||x_k||$, for any word~$w=x_1\ldots x_k$ in~$\Si^\ast$. In the rest of this article, and if there is no possible confusion,~$\Si$ denotes a signed alphabet.

\newpage

\section{The super plactic monoid}
\label{SS:SuperPlacticMonoid}

In this section, we recall from~\cite{LaScalaNardozzaSenato06} the structure of super Young tableaux and insertion algorithms on super tableaux. 
We also recall the notion of super plactic monoid and the cross-section property for this monoid and we deduce properties of the insertion product on super tableaux.

\subsection{Super Young tableaux}
\label{SS:SuperYoungTableaux}
A \emph{partition} of a positive integer~$n$, denoted by~$\lambda\vdash n$, is a weakly decreasing sequence~$\lambda=(\lambda_1,\ldots,\lambda_k)\in\mathbb{N}^k$ such that~$ \sum\lambda_i = n$. The integer~$k$ is called \emph{number of parts} or \emph{height} of~$\lambda$. 
Denote by~$\Pr_n$ the set of partitions of a positive integer~$n$, and set~$\Pr=\underset{n\in\mathbb{N}}{\bigcup\Pr_n}$.
The \emph{(Ferrers)--Young diagram} of a partition~$\lambda=(\lambda_1,\ldots,\lambda_k)$ is the set
\[
\Yr(\lambda)\;:=\;\big\{ (i,j)\;\big|\; 1\leq i\leq k, 1\leq j\leq \lambda_i \big\},
\]
that can be represented by a diagram by drawing a box for each pair~$(i,j)$. For instance, the Young diagram for~$\lambda=(3,3,1,1)$ is the following
\begin{equation}
\label{Ex:YoungtableauExample}
\ytableausetup{smalltableaux}
\begin{ytableau}
\empty& \empty&\empty \\
\empty&\empty&\empty\\
\empty&\none\\
\empty&\none
\end{ytableau}.
\end{equation}
The transposed diagram~$\{(j,i)\;\big|\; (i,j)\in\Yr(\lambda)\}$ defines another partition~$\widetilde{\lambda}\vdash n$, called the \emph{conjugate partition} of~$\lambda$, whose parts are the lengths of the columns of $\Yr(\lambda)$.
Let~$\lambda\vdash n$ be a partition. A \emph{super semistandard Young tableau}, or \emph{super tableau} for short, over~$\Si$ is a pair $t:=(\lambda,\T)$ where $\T:\Yr(\lambda)\to \Si$ is a map satisfying the following two conditions:
\begin{enumerate}[\bf i)]
\item $\T(i,j)\leq \T(i,j+1)$, with $\T(i,j) =\T(i,j+1)$ only if $||\T(i,j)||=0$,
\item $\T(i,j)\leq \T(i+1,j)$, with $\T(i,j) =\T(i+1,j)$ only if $||\T(i,j)||=1$.
\end{enumerate} 
We will call~$\Yr(\lambda)$,~$\T$ and~$\lambda$, the \emph{frame}, the \emph{filing} and the \emph{shape} of the super tableau~$t$ respectively.
When the map~$\T$ is injective, we say that~$t$ is a \emph{standard Young tableau}, or \emph{standard tableau} for short, over~$\Si$. 
Denote by~$\emptyset$ the empty super tableau over~$\Si$.
Denote by~$\YoungT(\Si)$  the set of all super tableaux over~$\Si$ and by~$\YoungT(\Si,\lambda)$ the set of all super tableaux  of shape~$\lambda$ over~$\Si$. Note that the usual notion of semistandard tableaux of type~A,~\cite{Fulton97}, is recovered when~$||.||$ has a constant value. More precisely, when $\Si=\Si_0$ (resp.~$\Si=\Si_1$), we obtain row-strict (resp.~column-strict) semistandard tableaux. 

Let~$\lambda$ and~$\mu$ be in~$\Pr$ with heights~$k$ and~$l$ respectively such that~$l\leq k$. We write~$\mu\subseteq\lambda$ if~$\mu_i\leq \lambda_i$, for any~$i$, that is the Young diagram of~$\mu$ is contained in that of~$\lambda$. In this case, the set
\[
\Yr(\lambda/\mu):=\big\{(i,j)\;\big|\; 1\leq i\leq k, \mu_i< j\leq \lambda_i \big\}
\]
is called a \emph{skew diagram}~$\lambda/\mu$ or a \emph{skew shape}. We denote~$\lambda/0 :=\lambda$ the skew shape corresponding to the Young diagram~$\Yr(\lambda)$.
For instance, the following diagram
\[
\scalebox{0.8}{
\ytableausetup{smalltableaux}
\begin{ytableau}
\none& \none&\none&\empty&\empty \\
\none&\none&\none&\empty\\
\none&\empty&\empty&\empty\\
\empty&\empty&\empty\\
\empty&\none
\end{ytableau}}
\]
corresponds to  the skew shape~$(5,4,4,3,1)/(3,3,1)$.
The \emph{super Young lattice} is a partially ordered set, denoted by~$(\Pr,\subseteq)$, formed by all integer partitions ordered by inclusion of their Young diagrams.
Let~$\lambda/\mu$ be a skew shape.
A \emph{super semistandard skew tableau}, or \emph{super skew tableau} for short,  of \emph{shape}~$\lambda/\mu$ and \emph{frame}~$\Yr(\lambda/\mu)$ over~$\Si$ is a pair $S:=(\lambda/\mu,\U)$ where $\U:\Yr(\lambda/\mu)\to \Si$ is a map satisfying the following two conditions:
\begin{enumerate}[\bf i)]
\item $\U(i,j)\leq \U(i,j+1)$, with $\U(i,j) =\U(i,j+1)$ only if $||\U(i,j)||=0$,
\item $\U(i,j)\leq \U(i+1,j)$, with $\U(i,j) =\U(i+1,j)$ only if $||\U(i,j)||=1$.
\end{enumerate} 
Denote by~$\SkewT(\Si)$  the set of all super skew tableaux over~$\Si$ and by~$\SkewT(\Si,\lambda/\mu)$ the set of all super skew tableaux of shape~$\lambda/\mu$ over~$\Si$. Note that we will identify the set  of super tableaux of shape~$\lambda$ with the set of super skew tableaux of shape~$\lambda/0$.

Denote by
\[
R_{row}:\SkewT(\Si)\to \Si^\ast
\]  
the \emph{reading map} that reads a super tableau row-wise from bottom to top and from left to right. For instance, consider the alphabet~$\Si=\{1,2,3,4,5\}$ with signature given by~$\Si_0=\{1,2,4\}$ and~$\Si_1$ defined consequently. 
The following diagram is a super tableau over~$\Si$ corresponding to the frame~\eqref{Ex:YoungtableauExample}:
\begin{equation}
\label{Ex:tableauExample}
\raisebox{-0.45cm}{$t\;
=
\;$}
\ytableausetup{smalltableaux}
\begin{ytableau}
1& 1&2 \\
3&4&4\\
5\\
5
\end{ytableau}
\qquad
\raisebox{-0.45cm}{$\text{ with }~R_{row}(t)=55344112.$}
\end{equation}

A \emph{row} (resp. \emph{column}) is a word~$x_1\ldots x_k$ in~$\Si^\ast$ such that~$x_{i}\leq x_{i+1}$ (resp.~$x_{i+1}\leq x_{i}$) with~$x_{i}= x_{i+1}$ only if~$||x_i||=0$ (resp.~$||x_i||=1$). In other words, a row (resp. column) is the $R_{row}$-reading of a super tableau whose shape is a row (resp. column).

\subsection{Insertion on super tableaux}
\label{SS:InsertionsOnSuperTableaux}

Recall from~\cite{LaScalaNardozzaSenato06} the right and left insertion algorithms on~$\YoungT(\Si)$ that insert an element~$x$ in~$\Si$ into a super tableau~$t$ of~$\YoungT(\Si)$.
The \emph{right (or row) insertion}, denoted by~$\insr{}$, computes a super tableau $t\insr{} x$ as follows. If~$x\in\Si_o$ (resp.~$x\in\Si_1$) is at least as large as (resp. larger than) the last element of the top row of $t$, then put~$x$ in a box to the right of this row. Otherwise, let~$y$ be the smallest element of the top row of~$t$ such that~$y> x$ (resp.~$y\geq x$). Then~$x$ replaces~$y$ in this row and~$y$ is bumped into the next row where the process is repeated. The algorithm terminates when the element which is bumped is at least as large as (resp. larger than) the last element of the next row. Then it is placed in a box  at the right of that row. For instance, consider~$\Si=\mathbb{N}$ with signature given by~$\Si_0$ the set of even numbers and~$\Si_1$ defined consequently. The four steps to compute
$\big(\raisebox{0.1cm}{\ytableausetup{smalltableaux} 
\begin{ytableau}
1& 2 &2&3\\
1 & 3&4 \\
3
\end{ytableau}}\insr{} 2 \big)$ are: 
\begin{equation}
\label{Eq:RightInsertionExample}
\ytableausetup{smalltableaux}
\begin{ytableau}
1& 2 &2&*(cyan)3&  \none& \none[\insr{}]&\none& *(red) 2\\
1 & 3&4 \\
3
\end{ytableau}
\qquad
\underset{\fl}
\quad
\ytableausetup{smalltableaux}
\begin{ytableau}
1&2&2&*(red) 2 \\
1 & *(green) 3&4& \none& \none[\insr{}]&\none& *(cyan) 3\\
3
\end{ytableau}
\qquad
\underset{\fl}
\quad
\ytableausetup{smalltableaux}
\begin{ytableau}
1&2&2& *(red) 2  \\
1 & *(cyan) 3&4 \\
*(yellow) 3 & \none&  \none[\insr{}]&\none& *(green) 3
\end{ytableau}
\qquad
\underset{\fl}
\quad
\ytableausetup{smalltableaux}
\begin{ytableau}
1&2&2& *(red) 2  \\
1 & *(cyan) 3&4 \\
*(green) 3\\
\none & \none&  \none[\insr{}]&\none & *(yellow) 3
\end{ytableau}
\qquad
\underset{\fl}
\quad
\ytableausetup{smalltableaux}
\begin{ytableau}
1&2&2& *(red) 2  \\
1 & *(cyan) 3&4 \\
*(green) 3\\
*(yellow) 3 
\end{ytableau}
\end{equation}

The \emph{left (or column) insertion}, denoted by~$\insl{}$, computes a super tableau~$x\insl{} t$ as follows. If~$x\in\Si_o$ (resp.~$x\in\Si_1$) is larger than (resp. at least as large as) the bottom element of the
leftmost column of~$t$, then put~$x$ in a box to the bottom of this column. Otherwise, let~$y$ be the smallest element of the leftmost column of~$t$ such that~$y\geq x$ (resp.~$y>x$). Then~$x$ replaces~$y$ in this column and~$y$ is bumped into the next column where the process is repeated. The algorithm terminates when the element which is bumped is greater than (resp. at least as large as) all the elements of the next column. Then it is placed in a box at the bottom of that column. For instance, consider~$\Si=\mathbb{N}$ with signature given by~$\Si_0$ the set of even numbers and~$\Si_1$ defined consequently. The five steps to compute
$\big(1 \insl{} \raisebox{0.1cm}{\ytableausetup{smalltableaux}
\begin{ytableau}
1& 2 &5&6\\
1 &4 &5 \\
2
\end{ytableau}}\,\big)$ are:
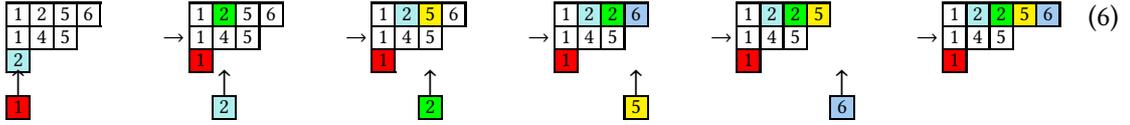
\begin{equation}
\label{Eq:LeftInsertionExample}
\ytableausetup{smalltableaux}
\begin{ytableau}
1&  2&5&6 \\
 1 & 4&5 \\
*(cyan)2\\
\none[\uparrow]\\
*(red) 1
\end{ytableau}
\qquad
\underset{\fl}
\quad
\ytableausetup{smalltableaux}
\begin{ytableau}
1&  *(green) 2&5&6  \\
 1 &4& 5 \\
*(red)1\\
\none &\none[\uparrow]\\
\none & *(cyan) 2
\end{ytableau}
\qquad
\underset{\fl}
\quad
\ytableausetup{smalltableaux}
\begin{ytableau}
1&  *(cyan) 2 &*(yellow) 5&6 \\
1 & 4&5 \\
*(red) 1\\
\none &\none &\none[\uparrow]\\
\none & \none & *(green) 2
\end{ytableau}
\qquad
\underset{\fl}
\quad
\ytableausetup{smalltableaux}
\begin{ytableau}
1&  *(cyan) 2 & *(green) 2&*(BlueD)6 \\
1 &4& 5 \\
*(red) 1\\
\none&\none &\none &\none[\uparrow]\\
\none&\none & \none & *(yellow) 5
\end{ytableau}
\qquad
\underset{\fl}
\quad
\ytableausetup{smalltableaux}
\begin{ytableau}
1&  *(cyan) 2 & *(green) 2& *(yellow) 5 \\
 1 & 4&5 \\
*(red)1\\
\none&\none &\none &\none &\none[\uparrow]\\
\none&\none & \none &\none & *(BlueD) 6
\end{ytableau}
\qquad
\underset{\fl}
\quad
\ytableausetup{smalltableaux}
\begin{ytableau}
1&  *(cyan) 2 & *(green) 2& *(yellow) 5 & *(BlueD) 6\\
1 & 4&5 \\
*(red) 1
\end{ytableau}
\end{equation}

Note that when~$\Si=\Si_0 = [n]$, the right (resp. left) insertion corresponds to Schensted's right (resp.~left) insertion introduced in~\cite{Schensted61} on row-strict semistandard tableaux over~$[n]$. 
For any word $w=x_1\ldots x_k$  over~$\Si$, denote by~$C_{\YoungT(\Si)}(w)$ the super tableau obtained from~$w$ by inserting its letters iteratively from left to right using the right insertion starting from the empty super tableau:
\[
C_{\YoungT(\Si)}(w) 
\;
:=
\;
(\emptyset \insr{} w)
\;
=
\;((\ldots(\emptyset \insr{} x_1)  \insr{} \ldots )\insr{} x_k).
\]
Note that for any super tableau~$t$ in~$\YoungT(\Si)$, the equality~$C_{\YoungT(\Si)}(R_{row}(t))=t$ holds in~$\YoungT(\Si)$,~\cite{LaScalaNardozzaSenato06}. 

Define now an internal product~$\star_{\YoungT(\Si)}$ on $\YoungT(\Si)$ by setting 
\[
t \star_{\YoungT(\Si)} t' :=  (t\insr{} R_{row}(t'))
\]
for all $t,t'$ in $\YoungT(\Si)$.
By definition the relations $t\star_{\YoungT(\Si)} \emptyset = t$ and $\emptyset \star_{\YoungT(\Si)} t = t$ hold, showing that the product~$\star_{\YoungT(\Si)}$ is unitary with respect to~$\emptyset$.

\subsection{The super plactic monoid}
\label{SS:SuperPlacticMonoid}

The \emph{super plactic monoid} over~$\Si$, denoted by~$\P(\Si)$,  is the quotient of the free monoid~$\Si^\ast$ by the congruence generated by the family of \emph{super Knuth-like relations}~\eqref{Eq:SuperKnuthRelations},~\cite{LaScalaNardozzaSenato06}. This congruence, denoted by~$\sim_{\P(\Si)}$, is called the \emph{super plactic congruence}.
We say that two words over~$\Si$ are \emph{super plactic equivalent} if one can be transformed into the other under the relations~\eqref{Eq:SuperKnuthRelations}.
Since the relations~(\ref{Eq:SuperKnuthRelations}) are $\Zb_2$-homogeneous, we have that~$\P(\Si)$ is a supermonoid. Moreover, for any~$w$ in~$\Si^\ast$, the words~$w$ and~$R_{row}(C_{\YoungT(\Si)}(w))$ are super plactic equivalent,~\cite{LaScalaNardozzaSenato06}. Note finally that the set~$\YoungT(\Si)$ satisfies the \emph{cross-section property} for~$\sim_{\P(\Si)}$, that is, for all~$w$ and~$w'$ in~$\Si^\ast$,~\cite{LaScalaNardozzaSenato06}:
\begin{equation}
\label{Eq:CrossSectionProperty}
w \sim_{\P(\Si)} w' \quad \text{ if and only if}\quad C_{\YoungT(\Si)}(w) = C_{\YoungT(\Si)}(w').
\end{equation}

As a consequence of the cross-section property, we deduce  that the product~$\star_{\YoungT(\Si)}$  is associative and the following equality
\[
y \insl{} (t \insr{} x)
\: = \:
(y \insl{} t) \insr{} x
\]
holds in~$\YoungT(\Si)$, for all~$t$ in~$\YoungT(\Si)$ and $x,y$ in~$\Si$.  In particular, for any word~$w=x_1\ldots x_k$ in~$\Si^\ast$ the super tableau~$C_{\YoungT(\Si)}(w)$ can be also computed by inserting its elements iteratively from right to left using the left insertion starting from the empty super  tableau:
\[
C_{\YoungT(\Si)}(w)
\; =\;
(w\insl{} \emptyset)
\; :=\;
(x_1 \insl{} ( \ldots\insl{}  (x_k\insl{} \emptyset)\ldots)).
\]
Note that the associativity of the product~$\star_{\YoungT(\Si)}$ is deduced in Corollary~\ref{C:AssociativityInsertionProductTableau} using the properties of the super jeu de taquin introduced in the next section.

Let~$w$ be in~$\Si^\ast$. For any~$k\geq 0$, denote by~$l_{k}(w)$ (resp.~$\widetilde{l}_k(w)$) the maximal number which can be obtained as the sum of the lengths of~$k$ rows (resp.~columns) that are disjoint subsequences of~$w$. The integers~$l_{k}(w)$ and~$\widetilde{l}_k(w)$ are the super analogues of \emph{Greene's row and columns invariants},~\cite{Greene74}. For any~$w'$ in~$\Si^\ast$, if~$w\sim_{\P(\Si)}w'$ then~$l_{k}(w)= l_{k}(w')$ for any~$k$,~\cite{LaScalaNardozzaSenato06}. Moreover, let~$\lambda=(\lambda_1,\ldots,\lambda_k)$ be the shape of~$C_{\YoungT(\Si)}(w)$  and~$\widetilde{\lambda}=(\widetilde{\lambda}_1,\ldots, \widetilde{\lambda}_l)$ be the conjugate partition of~$\lambda$. For any~$k\geq 0$, we have~$l_{k}(w) = \lambda_1 +\ldots +\lambda_k$  and~$\widetilde{l}_k(w)= \widetilde{\lambda}_1 + \ldots+ \widetilde{\lambda}_k$,~\cite{LaScalaNardozzaSenato06}.
Then, the equality~$\lambda_k=l_{k}(w)-l_{k-1}(w)$ holds, for any~$k$. Hence, we obtain the following result.

\begin{lemma}
\label{L:GreeneInvariants}
Let~$w$ and~$w'$ be two words over~$\Si$. If the equality~$l_{k}(w)= l_{k}(w')$ holds for any~$k$, then the super tableaux~$C_{\YoungT(\Si)}(w)$ and~$C_{\YoungT(\Si)}(w')$ have the same shape. 
\end{lemma}

\section{The super Jeu de taquin}
\label{S:SuperJeudeTaquin}

Following the approach given in~\cite{Fulton97} for the non-signed case, we introduce the super jeu de taquin procedure which transforms super skew tableaux into super  tableaux and we  show that it is confluent. We relate this procedure to insertion algorithms  and we show how insertion on super tableaux  can be performed by taquin.

\subsection{The forward sliding}
\label{SS:ForwardSlidings}

Let~$\lambda/\mu$ be a skew shape. A super skew tableau of shape~$\lambda/\mu$ which is not a super tableau has at least one inner corner. An \emph{inner corner} is a box in the diagram~$\Yr(\mu)$ such that the boxes below and to the right do not belong to~$\Yr(\mu)$. An \emph{outer corner} is a box such that neither box below or to the right is in~$\Yr(\lambda/\mu)$. Note that in some cases inner corners can be also outer corners. For instance, consider~$\Si=\mathbb{N}$ with signature given by~$\Si_0$ the set of even numbers and~$\Si_1$ defined consequently. The red box in the following super skew tableau of shape~$(6,5,5,3,1)/(4,4,4,3)$ is an inner corner, the blue one is an inner and outer corner, and the green one is an outer corner that is not an inner corner:
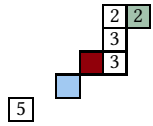
\[
{\ytableausetup{smalltableaux}
\begin{ytableau}
\none& \none&\none&\none&2&*(GreenD)2\\
\none&\none&\none&\none&3\\
\none&\none&\none&\empty*(RedD)&3\\
\none&\none&\empty*(BlueD)\\
5
\end{ytableau}}
\]
A \emph{sliding operation} belongs to one of the following operations:
\begin{enumerate}[\bf i)]
\item \emph{vertical sliding:}~$\ytableausetup{mathmode, smalltableaux}
{\begin{ytableau}
 *(GreenD)&y\\
   x
\end{ytableau}}
 \;\rightarrow\;
{ \begin{ytableau}
x& y\\
 *(GreenD)&\none
\end{ytableau}}$,\;
for any  $x\leq y$ with $x=y$ only if~$||x||=0$,
\item \emph{horizontal sliding:}~${\begin{ytableau}
 *(GreenD)&x\\
y
\end{ytableau}}
\;\rightarrow\;
{\begin{ytableau}
x& *(GreenD)\\
y&\none 
\end{ytableau}}$,\; for any~$x\leq y$ with $x=y$ only if~$||x||=1$.
\end{enumerate} 
Note that if~$x$ or~$y$ in~{\bf i)} and~{\bf ii)} are empty, the following operations are performed:
\[
{\begin{ytableau}
 *(GreenD)&\empty\\
x
\end{ytableau}}
\;\rightarrow\;
{\begin{ytableau}
x&\empty\\
 *(GreenD)&\none 
\end{ytableau}},
\qquad 
{\begin{ytableau}
*(GreenD) &x\\
\empty&\none
\end{ytableau}}
\;\rightarrow\;
{\begin{ytableau}
x& *(GreenD)\\
\empty&\none 
\end{ytableau}},
\qquad
{\begin{ytableau}
*(GreenD) &\empty\\
\empty&\none
\end{ytableau}}
\;\rightarrow\;
{\begin{ytableau}
\empty& *(GreenD)\\
\empty&\none 
\end{ytableau}}\quad \text{ for any } x.
\]
A \emph{forward sliding} is a sequence of sliding operations starting from a super skew tableau and one of its inner corners, and moving the empty box until it becomes an outer corner.
The \emph{super jeu de taquin} on a super skew tableau~$S$ consists in applying successively the forward sliding algorithm starting from~$S$ until we get a diagram without inner corners. The resulted diagram, denoted  by $\Rec(S)$, is called the \emph{rectification} of~$S$.
Two super  skew  tableaux~$S$ and~$S'$ are \emph{super jeu de taquin equivalent} if and only if~$S$ can be obtained from~$S'$ by a sequence of sliding operations.

\subsubsection{Example}
\label{SS:ExampleJeuDeTaquin}
Consider~$\Si=\mathbb{N}_{n>0}$ with signature given by~$\Si_0$ the set of even numbers and~$\Si_1$ defined consequently. 
The super jeu de taquin on the following super skew tableau~$S$ applies six occurrences of forward sliding:
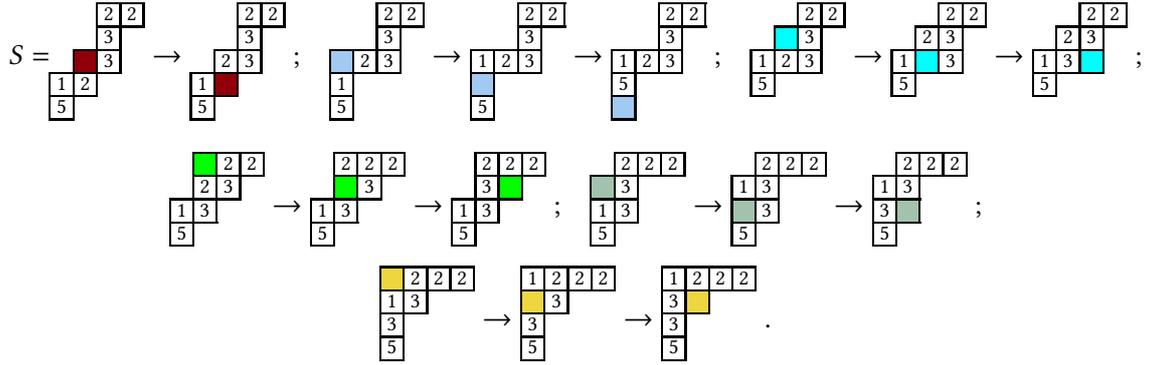
\[
\raisebox{-0.5cm}{$S=$}
{
\ytableausetup{smalltableaux}
\begin{ytableau}
\none&\none&2&2\\
\none&\none&3\\
\none&\empty*(RedD)&3\\
1&2\\
5
\end{ytableau}
\;
\raisebox{-0.5cm}{$\rightarrow$}}
\;
{
\ytableausetup{smalltableaux}
\begin{ytableau}
\none&\none&2&2\\
\none&\none&3\\
\none&2&3\\
1&\empty*(RedD)\\
5
\end{ytableau}}
\;
\raisebox{-0.5cm}{$;$}
\quad
{
\ytableausetup{smalltableaux}
\begin{ytableau}
\none&\none&2&2\\
\none&\none&3\\
\empty*(BlueD)&2&3\\
1\\
5
\end{ytableau}
\;
\raisebox{-0.5cm}{$\rightarrow$}}
\;
{\ytableausetup{smalltableaux}
\begin{ytableau}
\none&\none&2&2\\
\none&\none&3\\
1&2&3\\
\empty*(BlueD)\\
5
\end{ytableau}
\;
\raisebox{-0.5cm}{$\rightarrow$}}
\;
{\ytableausetup{smalltableaux}
\begin{ytableau}
\none&\none&2&2\\
\none&\none&3\\
1&2&3\\
5\\
\empty*(BlueD)
\end{ytableau}}
\;
\raisebox{-0.5cm}{$;$}
\quad
{\ytableausetup{smalltableaux}
\begin{ytableau}
\none&\none&2&2\\
\none&\empty*(Blue)&3\\
1&2&3\\
5
\end{ytableau}
\;
\raisebox{-0.5cm}{$\rightarrow$}}
\;
{\ytableausetup{smalltableaux}
\begin{ytableau}
\none&\none&2&2\\
\none&2&3\\
1&\empty*(Blue)&3\\
5
\end{ytableau}
\;
\raisebox{-0.5cm}{$\rightarrow$}}
\;
{\ytableausetup{smalltableaux}
\begin{ytableau}
\none&\none&2&2\\
\none&2&3\\
1&3&\empty*(Blue)\\
5
\end{ytableau}}
\;
\raisebox{-0.5cm}{$;$}
\]
\[
{\ytableausetup{smalltableaux}
\begin{ytableau}
\none&\empty*(GreenL)&2&2\\
\none&2&3\\
1&3\\
5
\end{ytableau}
\;
\raisebox{-0.5cm}{$\rightarrow$}}
\;
{
\ytableausetup{smalltableaux}
\begin{ytableau}
\none&2&2&2\\
\none&\empty*(GreenL)&3\\
1&3\\
5
\end{ytableau}
\;
\raisebox{-0.5cm}{$\rightarrow$}}
\;
{\ytableausetup{smalltableaux}
\begin{ytableau}
\none&2&2&2\\
\none&3&\empty*(GreenL)\\
1&3\\
5
\end{ytableau}}
\;
\raisebox{-0.5cm}{$;$}
\quad
{
\ytableausetup{smalltableaux}
\begin{ytableau}
\none&2&2&2\\
\empty*(GreenD)&3\\
1&3\\
5
\end{ytableau}
\;
\raisebox{-0.5cm}{$\rightarrow$}}
\;
{
\ytableausetup{smalltableaux}
\begin{ytableau}
\none&2&2&2\\
1&3\\
\empty*(GreenD)&3\\
5
\end{ytableau}
\;
\raisebox{-0.5cm}{$\rightarrow$}}
\;
{
\ytableausetup{smalltableaux}
\begin{ytableau}
\none&2&2&2\\
1&3\\
3&\empty*(GreenD)\\
5
\end{ytableau}}
\;
\raisebox{-0.5cm}{$;$}
\]
\[
{
\ytableausetup{smalltableaux}
\begin{ytableau}
\empty*(YellowD)&2&2&2\\
1&3\\
3\\
5
\end{ytableau}
\;
\raisebox{-0.5cm}{$\rightarrow$}}
\;
{
\ytableausetup{smalltableaux}
\begin{ytableau}
1&2&2&2\\
\empty*(YellowD)&3\\
3\\
5
\end{ytableau}
\;
\raisebox{-0.5cm}{$\rightarrow$}}
\;
{
\ytableausetup{smalltableaux}
\begin{ytableau}
1&2&2&2\\
3&\empty*(YellowD)\\
3\\
5
\end{ytableau}}
\;
\raisebox{-0.5cm}{$.$}
\]
Hence we obtain
\[
\raisebox{-0.5cm}{$\Rec(S)\;=\;$}
{
\ytableausetup{smalltableaux}
\begin{ytableau}
1&2&2&2\\
3\\
3\\
5
\end{ytableau}}
\]

\begin{lemma}
\label{L:RectificationTableau}
Let~$S$ be in~$\SkewT(\Si)$. The rectification~$\Rec(S)$ of~$S$ is a super tableau.
\end{lemma}

\begin{proof}
It is sufficient to prove that after each sliding operation the rows (resp. columns) in the super skew tableau remain increasing from left to right (resp. top to bottom) with respect to~$\Si_0$ (resp.~$\Si_1$). Suppose that the inner corner has the following position:
\[
\ytableausetup{mathmode, smalltableaux}
\begin{ytableau}
\none &r_1&r_2\\ 
r_3&*(GreenD)&y\\
 r_4&  x
\end{ytableau}
\]
where~$r_1\leq r_2\leq y$ with~$r_1=r_2$ only if~$||r_1||=0$ and~$r_2=y$ only if~$||y||=1$, and $r_3\leq r_4\leq x$ with~$r_3=r_4$ only if~$||r_3||=1$ and~$r_4=x$ only if~$||x||=0$. Suppose first that~$x\leq y$ with~$x=y$ only if~$||y||=0$. Then we perform the following sliding operation:
\begin{equation}
\label{Eq:Rectification1}
\ytableausetup{mathmode, smalltableaux}
\begin{ytableau}
\none &r_1&r_2\\ 
r_3&*(GreenD)&y\\
 r_4&  x
\end{ytableau}
 \;\rightarrow\;
 \begin{ytableau}
\none &r_1&r_2\\ 
r_3&x&y\\
 r_4 &*(GreenD)
\end{ytableau}
\end{equation}
with~$r_3<x\leq y$ and~$x=y$ only if~$||y||=0$ showing that the row containing~$r_3$,~$x$ and~$y$ is increasing from left to right with respect to~$\Si_0$. We proceed in the same way for the case where~$y$ is empty.
Suppose now that~$y\leq x$ with~$x=y$ only if~$||y||=1$. Then we perform the following sliding operation:
\begin{equation}
\label{Eq:Rectification2}
\ytableausetup{mathmode, smalltableaux}
\begin{ytableau}
\none &r_1&r_2\\ 
r_3&*(GreenD)&y\\
 r_4&  x
\end{ytableau}
 \;\rightarrow\;
 \begin{ytableau}
\none &r_1&r_2\\ 
r_3&y&*(GreenD)\\
 r_4 &x
\end{ytableau}
\end{equation}
with~$r_1<y\leq x$ and~$x=y$ only if~$||y||=1$ showing that the column containing~$r_1$,~$y$ and~$x$ is  increasing from top to bottom with respect to~$\Si_1$.
We proceed in the same way for the case where~$x$ is empty. Note finally that when~$x$ and~$y$ are both empty, the rows (resp. columns) in the initial super skew tableau are not changed by the sliding operation and remain increasing from left to right (resp. top to bottom) with respect to~$\Si_0$ (resp.~$\Si_1$), showing the claim.
\end{proof}

\subsection{The reverse sliding}
\label{SS:ReverseSlidings}
A \emph{reverse sliding operation} belongs to one of the following operations:
\begin{enumerate}[\bf i)]
\item \emph{vertical reverse sliding:}~$\ytableausetup{mathmode, smalltableaux}
\begin{ytableau}
 \none &y\\
   x&*(GreenD)
\end{ytableau}
 \;\rightarrow\;
 \begin{ytableau}
\none &*(GreenD)\\
   x&y
\end{ytableau}$,\;
for any  $x\leq y$ with $x=y$ only if~$||x||=0$,
\item \emph{horizontal reverse sliding:}~$
\begin{ytableau}
\none &x\\
y&*(GreenD)
\end{ytableau}
\;\rightarrow\;
{\begin{ytableau}
\none &x\\
*(GreenD)&y
\end{ytableau}}$,\; for any~$x\leq y$ with $x=y$ only if~$||x||=1$.
\end{enumerate}
A \emph{reverse sliding} is a sequence of reverse sliding operations starting from a skew tableau and one of its empty outer corners, and moving this box until it becomes an inner corner.

Note that if we apply the reverse sliding operations on the right sides of the forward transformations  in~(\ref{Eq:Rectification1}) and~(\ref{Eq:Rectification2}),  we obtain their left sides. Indeed, in the first transformation when~$x \leq y$ with~$x=y$ only if~$||y||=0$, we have~$r_4\leq x$ with~$r_4=x$ only if~$||x||=0$, so we choose the vertical reverse sliding, which leads us back to the left side of~(\ref{Eq:Rectification1}). In the second case, we have~$r_2 \leq y$ with~$r_2=y$ only if~$||y||=1$, so the  horizontal reverse  sliding  leads us back to the left side of~(\ref{Eq:Rectification2}).
Hence, starting from the resulting super skew tableau of a forward sliding, together with
the outer corner that was removed, and if we apply the reverse sliding, then we will obtain
the initial super skew tableau with the chosen inner corner.

\subsection{Compatibility with the super plactic congruence} 
The sliding operations are compatible with the super plactic congruence. Indeed, for any~$x\leq y\leq z$ with $x=y$ only if~$||y||=0$ and~$y=z$ only if~$||y||=1$, we have
\[
\raisebox{-0.1cm}{$S=\;$}
\ytableausetup{mathmode, smalltableaux}
\begin{ytableau}
\none& *(GreenD)&y\\
x&z
\end{ytableau}
 \;\rightarrow\;
{\begin{ytableau}
*(GreenD)&y\\
x&z
\end{ytableau}}
\;\rightarrow\;
{\begin{ytableau}
x&y\\
z
\end{ytableau}}
\raisebox{-0.1cm}{$\;=\Rec(S),$}
\]
with~$R_{row}(S)=xzy\sim_{\P(\Si)} zxy=R_{row}(\Rec(S))$. Moreover, for any  $x\leq y\leq z$ with $x=y$ only if~$||y||=1$ and~$y=z$ only if~$||y||=0$, we have
\[ 
{\raisebox{-0.1cm}{$S=\;$}
\begin{ytableau}
 \none& *(GreenD)&x\\
y&z&\none 
\end{ytableau}}
\;\rightarrow\;
{\begin{ytableau}
 *(GreenD)&x\\
y&z
\end{ytableau}}
\;\rightarrow\;
{\begin{ytableau}
x&z\\
y 
\end{ytableau}
\raisebox{-0.1cm}{$\;=\Rec(S),$}}
\]
with~$R_{row}(S)=yzx\sim_{\P(\Si)} yxz=R_{row}(\Rec(S))$.
More generally, we show the following result.

\begin{lemma}
\label{L:KnuthequivalenceJeudeTaquin}
Let~$S$ and~$S'$ be in~$\SkewT(\Si)$. If~$S$ and~$S'$ are super jeu de taquin equivalent, then the words~$R_{row}(S)$ and~$R_{row}(S')$ are super plactic equivalent.
\end{lemma}

\begin{proof}
If~$S$ and~$S'$ are related by a sequence of horizontal sliding operations, then  by definition of the reading map~$R_{row}$, the words~$R_{row}(S)$ and~$R_{row}(S')$ are equal. For a vertical sliding, suppose that the overlapping part of the two rows where the sliding operation occurs has the following form:
\[
\scalebox{0.6}{
\begin{tikzpicture}
\draw  (0,0) -- ++(0,2) --++(3.91,0) --++(0,-1) --++(-1,0)  --++(0,-1) --++(-2.91,0) -- cycle
node[xshift=1cm, yshift=1cm]{$\scalebox{1.5} {t}$}
node[xshift=1.5cm, yshift=2.2cm]{\scalebox{1} {$K$}}
node[xshift=2.2cm, yshift=2.25cm,draw]{$s_1$} 
node[xshift=2.7cm, yshift=2.25cm]{$\cdots$} 
node[xshift=3.2cm, yshift=2.25cm,draw]{$s_k$} 
node[xshift=3.7cm, yshift=2.25cm]{$x$} 
node[xshift=4.2cm, yshift=2.25cm,draw]{$z_1$} 
node[xshift=4.7cm, yshift=2.25cm]{$\cdots$} 
node[xshift=5.2cm, yshift=2.25cm,draw]{$z_l$}
node[xshift=2.2cm, yshift=2.75cm,draw]{$r_1$} 
node[xshift=2.7cm, yshift=2.75cm]{$\cdots$} 
node[xshift=3.2cm, yshift=2.75cm,draw]{$r_k$}  
node[xshift=4.2cm, yshift=2.75cm,draw]{$y_1$} 
node[xshift=4.7cm, yshift=2.75cm]{$\cdots$} 
node[xshift=5.2cm, yshift=2.75cm,draw]{$y_l$}
node[xshift=6cm, yshift=2.75cm]{\scalebox{1} {$L$}} ;
\draw (1,2) -- ++(0,0.5)-- ++(4.48,0)-- ++(0,-0.5)-- ++(-4.48,0)-- cycle;
\draw (2.915,3) -- ++(0,2) --++(6,0) --++(0,-1) --++(-1,0)  --++(0,-1) --++(-5,0)  -- cycle;
\draw (1.93,2.5) -- ++(0,0.5)-- ++(4.5,0)-- ++(0,-0.5)-- ++(-4.5,0)-- cycle
node[xshift=4cm, yshift=1.5cm]{\scalebox{1.5} {$t^\prime$}};
\draw[fill=GreenD]  (3.5,2.5) -- ++(0,0.5) --++(0.4,0) --++(0,-0.5) --++(-0.4,0) -- cycle;
\end{tikzpicture}
}
\quad
\raisebox{1cm}{$\rightarrow$}
\quad
\scalebox{0.6}{
\begin{tikzpicture}
\draw  (0,0) -- ++(0,2) --++(3.91,0) --++(0,-1) --++(-1,0)  --++(0,-1) --++(-2.91,0) -- cycle
node[xshift=1cm, yshift=1cm]{$\scalebox{1.5} {t}$}
node[xshift=1.5cm, yshift=2.2cm]{\scalebox{1} {$K$}}
node[xshift=2.2cm, yshift=2.25cm,draw]{$s_1$} 
node[xshift=2.7cm, yshift=2.25cm]{$\cdots$} 
node[xshift=3.2cm, yshift=2.25cm,draw]{$s_k$} 
node[xshift=3.7cm, yshift=2.75cm]{$x$} 
node[xshift=4.2cm, yshift=2.25cm,draw]{$z_1$} 
node[xshift=4.7cm, yshift=2.25cm]{$\cdots$} 
node[xshift=5.2cm, yshift=2.25cm,draw]{$z_l$}
node[xshift=2.2cm, yshift=2.75cm,draw]{$r_1$} 
node[xshift=2.7cm, yshift=2.75cm]{$\cdots$} 
node[xshift=3.2cm, yshift=2.75cm,draw]{$r_k$}  
node[xshift=4.2cm, yshift=2.75cm,draw]{$y_1$} 
node[xshift=4.7cm, yshift=2.75cm]{$\cdots$} 
node[xshift=5.2cm, yshift=2.75cm,draw]{$y_l$}
node[xshift=6cm, yshift=2.75cm]{\scalebox{1} {$L$}} ;
\draw (1,2) -- ++(0,0.5)-- ++(4.48,0)-- ++(0,-0.5)-- ++(-4.48,0)-- cycle;
\draw (2.915,3) -- ++(0,2) --++(6,0) --++(0,-1) --++(-1,0)  --++(0,-1) --++(-5,0)  -- cycle;
\draw (1.93,2.5) -- ++(0,0.5)-- ++(4.5,0)-- ++(0,-0.5)-- ++(-4.5,0)-- cycle
node[xshift=4cm, yshift=1.5cm]{\scalebox{1.5} {$t^\prime$}};
\draw[fill=GreenD]  (3.5,2) -- ++(0,0.5) --++(0.4,0) --++(0,-0.5) --++(-0.4,0) -- cycle;
\end{tikzpicture}
}
\]
such that~$R=r_1\ldots r_k$,~$T=s_1\ldots s_k$,~$Y=y_1\ldots y_l$,~$Z=z_1\ldots z_l$,~$K$ and~$L$ are rows,~$t$ and~$t'$ are in~$\YoungT(\Si)$, and~$r_i\leq s_i$ (resp.~$y_i\leq z_i$) with~$r_i=s_i$ (resp.~$y_i = z_i$) only if~$||r_i||=1$
(resp.~$||y_i||=1$) for all~\mbox{$1\leq i\leq k$} (resp.~$1\leq i\leq l$). We will show that the words~$R_{row}(t)R_{row}(K)T x ZRYR_{row}(L)R_{row}(t')$ and $R_{row}(t)R_{row}(K)T  ZRxYR_{row}(L)R_{row}(t')$ are super plactic equivalent. Since~$\sim_{\P(\Si)}$ is a congruence, it is sufficient to show that
\begin{equation}
\label{Eq:KnuthequivalenceJeudeTaquin}
T x ZRY
\;
\sim_{\P(\Si)}
\;
T ZR x Y.
\end{equation}
We will show~(\ref{Eq:KnuthequivalenceJeudeTaquin}) by induction on~$k$.
Suppose that~$k=0$. The equality~$C_{\YoungT(\Si)}(x ZRY)= C_{\YoungT(\Si)}(ZR x Y)$ holds in~$\YoungT(\Si)$ showing by the cross-section property~$(\ref{Eq:CrossSectionProperty})$ that the words~$x ZRY$ and~$ZR x Y$ are super plactic equivalent. Suppose that~(\ref{Eq:KnuthequivalenceJeudeTaquin}) holds for~$k-1$ and consider~$R=r_1\rem(R)$ and~$T=t_1\rem(T)$ with~$|\rem(R)| = |\rem(T)|=k-1$. By the induction hypothesis, we have
\begin{equation}
\label{Eq:KnuthequivalenceJeudeTaquin2}
\rem(T) x Z\rem(R)Y
\;
\sim_{\P(\Si)}
\;
\rem(T) Z\rem(R) x Y.
\end{equation}
The equality~$C_{\YoungT(\Si)}(t_1\rem(T)x Zr_1)= C_{\YoungT(\Si)}(t_1r_1\rem(T)x Z)$ holds in~$\YoungT(\Si)$ showing by~$(\ref{Eq:CrossSectionProperty})$ that
\begin{equation}
\label{Eq:KnuthequivalenceJeudeTaquin3}
t_1\rem(T)x Zr_1
\;
\sim_{\P(\Si)}
\;
t_1r_1\rem(T)x Z.
\end{equation}
Similarly, since the equality~$C_{\YoungT(\Si)}(t_1r_1\rem(T)Z)= C_{\YoungT(\Si)}(t_1\rem(T)Zr_1)$ holds in~$\YoungT(\Si)$, we have:
\begin{equation}
\label{Eq:KnuthequivalenceJeudeTaquin4}
t_1r_1\rem(T)Z
\;
\sim_{\P(\Si)}
\;
t_1\rem(T)Zr_1.
\end{equation}
Hence, we obtain the following equivalence
\[
\begin{array}{rl}
T x ZRY\; &=\; t_1\rem(T)x Zr_1\rem(R)Y\\
        &\sim_{\P(\Si)} t_1r_1\rem(T)x Z\rem(R)Y\qquad [ \text{by}~(\ref{Eq:KnuthequivalenceJeudeTaquin3})]\\
&
{\sim_{\P(\Si)}}\;t_1r_1\rem(T) Z\rem(R) x Y\qquad [ \text{by}~(\ref{Eq:KnuthequivalenceJeudeTaquin2})]\\
&
{\sim_{\P(\Si)}}\;t_1\rem(T)Zr_1\rem(R) x Y\qquad [ \text{by}~(\ref{Eq:KnuthequivalenceJeudeTaquin4})]\\
&
=\; TZR x Y
\end{array}
\]
showing the claim.
\end{proof}

\begin{lemma}
\label{L:RectificationUnicity}
 The rectification tableau~$\Rec(S)$ of a given super skew tableau~$S$ is the unique super tableau such that~$R_{row}(S)\sim_{\P(\Si)} R_{row}(\Rec(S))$. Moreover, the following property
\[
\Rec(S) = \Rec(S')\quad \text{ if and only if}\quad R_{row}(S) \sim_{\P(\Si)} R_{row}(S')
\]
holds, for all~$S, S'$  in~$\SkewT(\Si)$. 
\end{lemma}

\begin{proof}
By Lemmata~$\ref{L:RectificationTableau}$ and~$\ref{L:KnuthequivalenceJeudeTaquin}$ the rectification~$\Rec(S)$ of a given super skew tableau~$S$ is a super tableau such that~$R_{row}(S) \sim_{\P(\Si)} R_{row}(\Rec(S))$. Suppose that there exists a super tableau~$t$ that is also the rectification of~$S$. We obtain that~$R_{row}(t) \sim_{\P(\Si)} R_{row}(\Rec(S))$, showing by~(\ref{Eq:CrossSectionProperty}) that~$t=\Rec(S)$. 

Consider now two super skew tableaux~$S$ and~$S'$.
 Suppose  that~$\Rec(S) = \Rec(S')$. We have by Lemma~\ref{L:KnuthequivalenceJeudeTaquin} that~$R_{row}(S) \sim_{\P(\Si)} R_{row}(\Rec(S))$ and~$R_{row}(S') \sim_{\P(\Si)} R_{row}(\Rec(S'))$ showing that the words~$R_{row}(S)$ and~$R_{row}(S')$ are super plactic equivalent. Suppose finally  that~$R_{row}(S) \sim_{\P(\Si)} R_{row}(S')$. We obtain by Lemma~\ref{L:KnuthequivalenceJeudeTaquin} that~$R_{row}(\Rec(S)) \sim_{\P(\Si)} R_{row}(\Rec(S'))$ showing by~(\ref{Eq:CrossSectionProperty}) that~$\Rec(S) = \Rec(S')$.
\end{proof}

As a consequence, we obtain the following result.

\begin{theorem}[Confluence of the super jeu de taquin]
\label{T:ConfluenceJeuDeTaquin}
Starting  with a given super skew tableau, all choices of inner corners lead to the same rectified super tableau.
\end{theorem}

\subsection{Super jeu de taquin and insertion}
\label{SS:InsertionsSuperJeuDeTaquin}
Given two super  skew tableaux $S$ and~$S'$ in~$\SkewT(\Si)$ of shape $(\lambda_1,\ldots,\lambda_k)/(\lambda'_1,\ldots,\lambda'_{k'})$ and~$(\mu_1,\ldots,\mu_l)/(\mu'_1,\ldots,\mu'_{l'})$ respectively. We will denote by~$[S,S']$ the super skew tableau 
of shape
\[
(\mu_1+\lambda_1,\ldots,\mu_l+\lambda_1, \lambda_1,\ldots,\lambda_{k}) /(\mu'_1+\lambda_1,\ldots, \mu'_{l'}+\lambda_1, \lambda_1,\ldots, \lambda_1, \lambda'_1,\ldots, \lambda'_{k'}),
\]
obtained by concatenating $S'$ over the rightmost column of~$S$ as illustrated in the following diagrams:
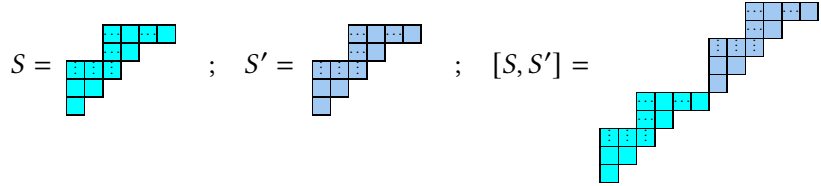
\[
\raisebox{-0.7cm}{$S=$}
\raisebox{-0.3cm}{
\scalebox{0.5}{
\ytableausetup{mathmode, smalltableaux, boxsize=1.2em}
\begin{ytableau}
\none& \none & \empty*(Blue)\cdots
&  \empty*(Blue)& \empty*(Blue)\cdots
&  \empty*(Blue) \\
\none & \none& \empty*(Blue)\cdots
&  \empty*(Blue) \\
\empty*(Blue)\vdots &\empty*(Blue)\vdots&\empty*(Blue)\vdots \\
\empty*(Blue) & \empty*(Blue) \\
\empty*(Blue)
\end{ytableau}
}}
\quad
\raisebox{-0.7cm}{$;$}
\quad
\raisebox{-0.7cm}{$S'=$}
\raisebox{-0.3cm}{
\scalebox{0.5}{
\ytableausetup{mathmode, smalltableaux, boxsize=1.2em}
\begin{ytableau}
\none& \none & \empty*(BlueD)\cdots
&  \empty*(BlueD)& \empty*(BlueD)\cdots
&  \empty*(BlueD) \\
\none & \none& \empty*(BlueD)\cdots
&  \empty*(BlueD) \\
\empty*(BlueD)\vdots &\empty*(BlueD)\vdots &\empty*(BlueD)\vdots  \\
\empty*(BlueD) & \empty*(BlueD) \\
\empty*(BlueD)
\end{ytableau}
}}
\quad
\raisebox{-0.7cm}{$;$}
\quad
\raisebox{-0.7cm}{$[S,S']=$}
\scalebox{0.5}{
\ytableausetup{mathmode, smalltableaux, boxsize=1.2em}
\begin{ytableau}
\none&\none&\none&\none&\none&\none&\none& \none & \empty*(BlueD)\cdots
&  \empty*(BlueD)& \empty*(BlueD)\cdots
&  \empty*(BlueD) \\
\none&\none&\none&\none&\none&\none&\none & \none& \empty*(BlueD)\cdots
&  \empty*(BlueD) \\
\none&\none&\none&\none&\none&\none&\empty*(BlueD)\vdots &\empty*(BlueD)\vdots &\empty*(BlueD)\vdots  \\
\none&\none&\none&\none&\none&\none&\empty*(BlueD) & \empty*(BlueD) \\
\none&\none&\none&\none&\none&\none&\empty*(BlueD)\\
\none& \none & \empty*(Blue)\cdots
&  \empty*(Blue)& \empty*(Blue)\cdots
&  \empty*(Blue) \\
\none & \none& \empty*(Blue)\cdots
&  \empty*(Blue) \\
\empty*(Blue)\vdots &\empty*(Blue)\vdots&\empty*(Blue)\vdots \\
\empty*(Blue) & \empty*(Blue) \\
\empty*(Blue)
\end{ytableau}
}
\]
In particular, if $t$ and~$t'$ are super tableaux in~$\YoungT(\Si)$ of shapes~$\lambda=(\lambda_1,\ldots,\lambda_k)$ and~$\mu=(\mu_1,\ldots,\mu_l)$ respectively, then the super skew tableau~$[t,t']$ is 
of shape
\[
(\mu_1+\lambda_1,\ldots,\mu_l+\lambda_1, \lambda_1,\ldots,\lambda_k)/(\lambda_1,\ldots,\lambda_1).
\]

Define the \emph{insertion product}~$\star_{\SkewT(\Si)}: \SkewT(\Si)\times \SkewT(\Si)\to \YoungT(\Si)$  by setting 
\[
S \star_{\SkewT(\Si)} S' :=  (\emptyset\insr{} R_{row}(S) R_{row}(S'))
\]
for all $S,S'$ in $\SkewT(\Si)$.

\begin{proposition}
\label{P:AssociativityInsertionProductSkew}
For all~$S$ and~$S'$ in~$\SkewT(\Si)$, the following equality
\begin{equation}
\label{Eq:JeuDeTaquinInsertion}
S\star_{\SkewT(\Si)} S' = \Rec([S,S'])
\end{equation}
holds in~$\YoungT(\Si)$. In particular, the insertion product~$\star_{\SkewT(\Si)}$ is associative. 
\end{proposition}

\begin{proof}
By definition of~$\star_{\SkewT(\Si)}$, the following equivalence~$R_{row}(S)R_{row}(S') \sim_{\P(\Si)} R_{row}(S\star_{\SkewT(\Si)} S')$  holds. Moreover, by Lemma~\ref{L:KnuthequivalenceJeudeTaquin}, the words $R_{row}([S,S']) = R_{row}(S)R_{row}(S')$ and $R_{row}(\Rec([S,S']))$ are super plactic equivalent,  showing that~$R_{row}(S\star_{\SkewT(\Si)} S')$ and~$R_{row}(\Rec([S,S']))$ are so. 
Hence, the equality~(\ref{Eq:JeuDeTaquinInsertion}) holds by the cross-section property~(\ref{Eq:CrossSectionProperty}).
In particular, by Theorem~\ref{T:ConfluenceJeuDeTaquin},  the following equality
\[
(S\star_{\SkewT(\Si)} S') \star_{\SkewT(\Si)} S'' 
\:=\: 
\Rec([[S , S' ],S'']) 
\:=\: 
S \star_{\SkewT(\Si)} (S' \star_{\SkewT(\Si)} S'')
\]
holds in~$\YoungT(\Si)$,  for all~$S,S',S''\in \SkewT(\Si)$,  showing that~$\star_{\SkewT(\Si)}$ is associative.
\end{proof}

By restriction of the insertion product on the set~$\YoungT(\Si)\times \YoungT(\Si)$, we deduce the following result.

\begin{corollary}
\label{C:AssociativityInsertionProductTableau}
For all~$t$ and~$t'$ in~$\YoungT(\Si)$, the following equality
\[
t\star_{\YoungT(\Si)} t' \;=\; \Rec([t,t'])
\]
holds in~$\YoungT(\Si)$, and  the insertion product~$\star_{\YoungT(\Si)}$ is associative. 
\end{corollary}

\subsubsection{Example}
The insertion
$\big(\raisebox{0.1cm}{\ytableausetup{smalltableaux} 
\begin{ytableau}
1& 2 &2&3\\
1 & 3&4 \\
3
\end{ytableau}}\insr{} 2 \big)$ computed in~(\ref{Eq:RightInsertionExample}) is  obtained  as follows: 
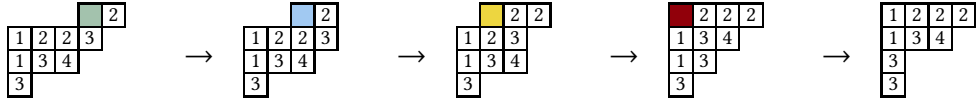
\[
\ytableausetup{smalltableaux}
\begin{ytableau}
\none& \none&\none& \empty *(GreenD) &2\\
1& 2 &2&3\\
1 & 3&4 \\
3
\end{ytableau}
\qquad
\raisebox{-0.5cm}{$\rightarrow$}
\quad
\begin{ytableau}
\none& \none&\empty *(BlueD)& 2\\
1& 2 &2&3\\
1 & 3&4 \\
3
\end{ytableau}
\qquad
\raisebox{-0.5cm}{$\rightarrow$}
\quad
\begin{ytableau}
\none&\empty *(YellowD) &2& 2\\
1& 2 &3\\
1 & 3&4 \\
3
\end{ytableau}
\qquad
\raisebox{-0.5cm}{$\rightarrow$}
\quad
\begin{ytableau}
\empty *(RedD)& 2 &2& 2\\
1 &3&4 \\
1 & 3\\
3
\end{ytableau}
\qquad
\raisebox{-0.5cm}{$\rightarrow$}
\quad
\begin{ytableau}
1 & 2 &2& 2\\
1 &3&4 \\
3\\
3
\end{ytableau}
\]
The  insertion
$\big(1 \insl{} \raisebox{0.1cm}{\ytableausetup{smalltableaux}
\begin{ytableau}
1& 2 &5&6\\
1 & 4&5 \\
2
\end{ytableau}}\,\big)$ computed in~(\ref{Eq:LeftInsertionExample}) is also recovered as follows by taquin:
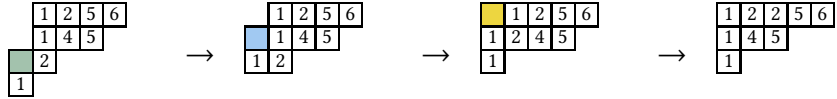
\[
\ytableausetup{smalltableaux}
\begin{ytableau}
\none&1& 2 &5&6\\
\none&1 & 4&5 \\
\empty *(GreenD)&2\\
1
\end{ytableau}
\qquad
\raisebox{-0.5cm}{$\rightarrow$}
\quad
\ytableausetup{smalltableaux}
\begin{ytableau}
\none&1& 2 &5&6\\
\empty *(BlueD)&1 & 4&5 \\
1&2
\end{ytableau}
\qquad
\raisebox{-0.5cm}{$\rightarrow$}
\quad
\ytableausetup{smalltableaux}
\begin{ytableau}
\empty *(YellowD)&1& 2 &5&6\\
1&2 & 4&5 \\
1
\end{ytableau}
\qquad
\raisebox{-0.5cm}{$\rightarrow$}
\quad
\ytableausetup{smalltableaux}
\begin{ytableau}
1&2& 2 &5&6\\
1& 4&5 \\
1
\end{ytableau}
\]

\subsection{Column reading on super tableaux}
Denote by~$R_{col}:\SkewT(\Si)\to \Si^\ast$ the map that reads a super tableau column-wise from bottom to top and from left to right. For instance, the~$R_{col}$-reading of the super  tableau~$t$ presented in~(\ref{Ex:tableauExample}) is~$R_{col}(t) = 55314142$.

\begin{proposition}
Let~$S$ be in~$\SkewT(\Si)$. The words $R_{col}(S)$ and~$R_{row}(S)$ are super  plactic equivalent.
\end{proposition}

\begin{proof}
Suppose that the super skew tableau~$S$ consists of~$k$ columns~$c_1,\ldots, c_k$.
By Theorem~\ref{T:ConfluenceJeuDeTaquin}, this super skew tableau can be obtained from~$[[\ldots[[c_1,c_2],c_3],\ldots],c_k]$ by applying a sequence of forward sliding that moves the columns to the top starting from~$c_{k-1}$ and ending by~$c_1$. 
Hence, we obtain by Lemma~\ref{L:KnuthequivalenceJeudeTaquin} the following equivalence:
\[
\begin{array}{rl}
R_{row}(S)& \sim_{\P(\Si)} R_{row}([[\ldots[[c_1,c_2],c_3],\ldots],c_k])\\
                &= R_{row}(c_1)\ldots R_{row}(c_k)\\
                & = R_{col}(c_1)\ldots R_{col}(c_k) \\
                &= R_{col}(S)
\end{array}
\]
showing the claim.
\end{proof}

\section{The super evacuation procedure}
\label{S:SuperJeuDeTaquinSuperEvacuation}

We introduce in this section the super evacuation procedure using the super jeu de taquin and we show its compatibility with the super plactic congruence. We follow for this aim, the approach given in~\cite{Fulton97} for the non-signed case.

\subsection{Super opposite alphabet}
\label{SS:SuperOppositeAlphabet}

Denote by~$\Si^{\mathrm{op}}$ the \emph{opposite alphabet} obtained from~$\Si$ by reversing its order. Denote by~$x^\ast$ the letter in~$\Si^{\mathrm{op}}$ corresponding to $x$ in~$\Si$ where~$||x^\ast||=0$ (resp.~$||x^\ast||=1$) if~$||x||=0$ (resp.~$||x||=1$). For all~$x, y$ in~$\Si$, we have~$x<y$ if and only if~$x^\ast>y^\ast$. For any word~$w=x_1\ldots x_k$ over~$\Si$, denote by~$w^\ast=x_k^\ast\ldots x_1^\ast$ the corresponding word over~$\Si^{\mathrm{op}}$. 
Then, for all~$v$ and~$w$ in~$\Si^\ast$, the equality~$(vw)^\ast = w^\ast v^\ast$ holds, inducing an anti-isomorphism between the free monoids over~$\Si^{\mathrm{op}}$ and~$\Si$.
Moreover, by identifying~$(\Si^{\mathrm{op}})^{\mathrm{op}}$ with~$\Si$,  the equality~$(w^\ast)^\ast = w$ holds, for any~$w$ in~$\Si^\ast$.

\begin{lemma}
\label{L:DualTableauxGreenePlactic}
Let~$v$ and~$w$ be words in~$\Si^\ast$. The following properties hold:
\begin{enumerate}[\bf i)]
\item $v\sim_{\P(\Si)} w$ if, and only, if~$v^\ast\sim_{\P(\Si^{\mathrm{op}})} w^\ast$.
\item $l_k(v)= l_k(v^\ast)$, for any~$k\geq 0$.
\end{enumerate}
\end{lemma}

\begin{proof}
Consider two words~$v$ and~$w$  in~$\Si^\ast$ such that~$v\sim_{\P(\Si)} w$. Suppose that~$v=pxzyq$ and~$w=pzxyq$ for all~$p,q\in\Si^\ast$ such that~$x\leq y\leq z$ in~$\Si$ with~$x=y$  only if  $||y|| = 0$  and~$y=z$ only if~$||y|| = 1$. We have~$v^\ast=q^{\ast}y^{\ast}z^{\ast}x^{\ast}p^{\ast}$ and~$w^{\ast}= q^{\ast}y^{\ast}x^{\ast}z^{\ast}p^{\ast}$ with~$z^{\ast}\leq y^{\ast}\leq x^{\ast}$ and~$x^{\ast}=y^{\ast}$  only if  $||y^{\ast}|| = 0$  and~$y^{\ast}=z^{\ast}$ only if~$||y^{\ast}|| = 1$, showing that~$v^\ast\sim_{\P(\Si^{\mathrm{op}})} w^\ast$. Suppose now that~$v=pyxzq$ and~$w=pyzxq$ for all~$p, q\in\Si^\ast$ such that~$x\leq y\leq z$ in~$\Si$ with~$x=y$  only if  $||y|| = 1$  and~$y=z$ only if~$||y|| = 0$. We have~$v^\ast=q^{\ast}z^{\ast}x^{\ast}y^{\ast}p^{\ast}$ and~$w^{\ast}= q^{\ast}x^{\ast}z^{\ast}y^{\ast}p^{\ast}$ with~$z^{\ast}\leq y^{\ast}\leq x^{\ast}$ and~$x^{\ast}=y^{\ast}$  only if  $||y^{\ast}|| = 1$  and~$y^{\ast}=z^{\ast}$ only if~$||y^{\ast}|| = 0$, showing that~$v^\ast\sim_{\P(\Si^{\mathrm{op}})} w^\ast$. Similarly, we show that if~$v^\ast\sim_{\P(\Si^{\mathrm{op}})} w^\ast$, then~$v\sim_{\P(\Si)} w$. 
Finally, for any~$k\geq 0$, any disjoint family of~$k$ rows that are subsequences of a word~$v$ over~$\Si$ gives by considering the mirror words of the rows in~$\Si^{\mathrm{op}}$, a family of~$k$ rows that are subsequences of~$v^\ast$, showing that~$l_k(v)= l_k(v^\ast)$.
 \end{proof}
 
\subsection{Involution on super tableaux}
\label{SS:InvolutionOnSuperTableaux}

Let~$t$ be in~$\YoungT(\Si)$.  An \emph{opposite tableau} in~$\YoungT(\Si^{\mathrm{op}})$ is constructed from~$t$ using the super jeu de taquin as follows. Consider an empty Young diagram with the same frame as~$t$. 
Remove the box containing the top-leftmost element~$x$ in~$t$, and perform a forward sliding on the resulted super skew tableau. 
We obtain a super tableau, denoted by~$t^\ast$, whose frame has one box removed from the frame of $t$. Put the letter $x^{\ast}$ in the initial empty Young diagram at the same place of the box that was removed from the frame of~$t$. Repeat the algorithm on~$t^\ast$ and continue until all the elements of~$t$ have been removed and the initial empty Young diagram has been filled by their corresponding letters in~$\Si^{\mathrm{op}}$. The result is denoted by~$t^{\mathrm{op}}$, and the procedure of construction~$t^{\mathrm{op}}$ from~$t$ is called the \emph{super evacuation}. This is the super analogue of the Schützenberger's evacuation procedure,~\cite{Schutzenberger63}.
For instance, consider the super tableau~$t$ presented in~(\ref{Ex:tableauExample}) and compute~$t^{\mathrm{op}}$.
\[
\raisebox{-0.45cm}{$t\;
=
\;$}
\ytableausetup{smalltableaux}
\begin{ytableau}
1& 1&2 \\
3&4&4\\
5\\
5
\end{ytableau}
\raisebox{-0.45cm}{$\quad ; \quad$}
\raisebox{-0.45cm}{$\big(\;$}
\ytableausetup{smalltableaux}
\begin{ytableau}
*(GreenD)& 1&2 \\
3&4&4\\
5\\
5
\end{ytableau}
\raisebox{-0.45cm}{$\;,\;$}
\raisebox{-0.45cm}{$t^\ast = \;$}
\begin{ytableau}
1&2&4 \\
3&4\\
5\\
5
\end{ytableau}
\raisebox{-0.45cm}{$\;,\;$}
\begin{ytableau}
\empty& \empty&\empty \\
\empty&\empty&1^\ast\\
\empty&\none\\
\empty&\none
\end{ytableau}
\raisebox{-0.45cm}{$\;\big)$}
\raisebox{-0.45cm}{$\quad ; \quad$}
\raisebox{-0.45cm}{$\big(\;$}
\ytableausetup{smalltableaux}
\begin{ytableau}
*(GreenD)&2&4 \\
3&4\\
5\\
5
\end{ytableau}
\raisebox{-0.45cm}{$\;,\;$}
\begin{ytableau}
2&4&4 \\
3\\
5\\
5
\end{ytableau}
\raisebox{-0.45cm}{$\;,\;$}
\begin{ytableau}
\empty& \empty&\empty \\
\empty&1^\ast&1^\ast\\
\empty&\none\\
\empty&\none
\end{ytableau}
\raisebox{-0.45cm}{$\;\big)$}
\raisebox{-0.35cm}{$\; ; $}
\]
\[
\raisebox{-0.45cm}{$\big(\;$}
\ytableausetup{smalltableaux}
\begin{ytableau}
*(GreenD)&4&4 \\
3\\
5\\
5
\end{ytableau}
\raisebox{-0.45cm}{$\;,\;$}
\begin{ytableau}
3&4&4 \\
5\\
5
\end{ytableau}
\raisebox{-0.45cm}{$\;,\;$}
\begin{ytableau}
\empty& \empty&\empty \\
\empty&1^\ast&1^\ast\\
\empty&\none\\
2^\ast
\end{ytableau}
\raisebox{-0.45cm}{$\;\big)$}
\raisebox{-0.45cm}{$\quad ; \quad$}
\raisebox{-0.35cm}{$\big(\;$}
\ytableausetup{smalltableaux}
\begin{ytableau}
*(GreenD)&4&4 \\
5\\
5
\end{ytableau}
\raisebox{-0.35cm}{$\;,\;$}
\begin{ytableau}
4&4 \\
5\\
5
\end{ytableau}
\raisebox{-0.35cm}{$\;,\;$}
\begin{ytableau}
\empty& \empty&3^\ast \\
\empty&1^\ast&1^\ast\\
\empty&\none\\
2^\ast
\end{ytableau}
\raisebox{-0.35cm}{$\;\big)$}
\raisebox{-0.35cm}{$\quad ; \quad$}
\raisebox{-0.35cm}{$\big(\;$}
\ytableausetup{smalltableaux}
\begin{ytableau}
*(GreenD)&4 \\
5\\
5
\end{ytableau}
\raisebox{-0.35cm}{$\;,\;$}
\begin{ytableau}
4\\
5\\
5
\end{ytableau}
\raisebox{-0.35cm}{$\;,\;$}
\begin{ytableau}
\empty& 4^\ast&3^\ast \\
\empty&1^\ast&1^\ast\\
\empty&\none\\
2^\ast
\end{ytableau}
\raisebox{-0.35cm}{$\;\big)$}
\raisebox{-0.35cm}{$\; ; $}
\]
\[
\raisebox{-0.35cm}{$\big(\;$}
\ytableausetup{smalltableaux}
\begin{ytableau}
*(GreenD) \\
5\\
5
\end{ytableau}
\raisebox{-0.35cm}{$\;,\;$}
\begin{ytableau}
5\\
5
\end{ytableau}
\raisebox{-0.35cm}{$\;,\;$}
\begin{ytableau}
\empty& 4^\ast&3^\ast \\
\empty&1^\ast&1^\ast\\
 4^\ast\\
2^\ast
\end{ytableau}
\raisebox{-0.35cm}{$\;\big)$}
\raisebox{-0.35cm}{$\quad ; \quad$}
\raisebox{-0.35cm}{$\big(\;$}
\raisebox{-0.35cm}{$\ytableausetup{smalltableaux}
\begin{ytableau}
*(GreenD) \\
5
\end{ytableau}$}
\raisebox{-0.35cm}{$\;,\;$}
\raisebox{-0.35cm}{$\begin{ytableau}
5
\end{ytableau}$}
\raisebox{-0.35cm}{$\;,\;$}
\begin{ytableau}
\empty& 4^\ast&3^\ast \\
5^\ast&1^\ast&1^\ast\\
 4^\ast\\
2^\ast
\end{ytableau}
\raisebox{-0.35cm}{$\;\big)$}
\raisebox{-0.45cm}{$\quad ; \quad$}
\raisebox{-0.35cm}{$\big(\;$}
\raisebox{-0.35cm}{$\ytableausetup{smalltableaux}
\begin{ytableau}
5
\end{ytableau}$}
\raisebox{-0.35cm}{$\;,\;$}
\raisebox{-0.35cm}{$\emptyset$}
\raisebox{-0.35cm}{$\;,\;$}
\begin{ytableau}
5^\ast& 4^\ast&3^\ast \\
5^\ast&1^\ast&1^\ast\\
 4^\ast\\
2^\ast
\end{ytableau}
\raisebox{-0.35cm}{$\;=\; t^{\mathrm{op}}\;\big).$}
\]

Let~$t$ be in~$\YoungT(\Si)$. Define~$t^{\wedge}$ to be the unique super tableau in~$\YoungT(\Si^{\mathrm{op}})$ such that the words $(R_{row}(t))^{\ast}$ and~$R_{row}(t^{\wedge})$ are congruent with respect to~$\sim_{\P(\Si^{\mathrm{op}})}$.

\begin{lemma}
\label{L:OppositeTableauLemma}
Let~$t$ be in~$\YoungT(\Si)$. The following properties hold:
\begin{enumerate}[\bf i)]
\item The map~$t\mapsto t^{\wedge}$ is an involution on~$\YoungT(\Si)$.
\item The super tableau~$t^{\wedge}$ has the same shape as~$t$.
\end{enumerate}
\end{lemma}

\begin{proof}
Consider~$t$ in~$\YoungT(\Si)$. By definition, the super tableau~$(t^{\wedge})^{\wedge}$ is the unique super tableau in~$\YoungT(\Si)$ such that~$(R_{row}(t^{\wedge}))^{\ast}\sim_{\P(\Si)}R_{row}\big( (t^{\wedge})^{\wedge}\big)$. Since the equality~$(w^\ast)^\ast=w$ holds for any~$w$ in~$\Si^\ast$, we obtain by Lemma~\ref{L:DualTableauxGreenePlactic} that~$R_{row}(t^{\wedge})\sim_{\P(\Si^{\mathrm{op}})}\big(R_{row}\big((t^{\wedge})^{\wedge}\big)\big)^\ast$ and then~$(R_{row}(t))^\ast \sim_{\P(\Si^{\mathrm{op}})}\big(R_{row}\big( (t^{\wedge})^{\wedge}\big)\big)^\ast$. Hence, by Lemma~\ref{L:DualTableauxGreenePlactic}, the equivalence~$R_{row}(t)\sim_{\P(\Si)}R_{row}\big( (t^{\wedge})^{\wedge}\big)$ holds, showing by the cross-section property~\eqref{Eq:CrossSectionProperty}, that the map~$t\mapsto t^{\wedge}$ is an involution.
Moreover, since the words~$(R_{row}(t))^{\ast}$ and~$R_{row}(t^{\wedge})$ are congruent with respect to~$\sim_{\P(\Si^{\mathrm{op}})}$, the following equality~$l_k(R_{row}(t^{\wedge})) = l_k((R_{row}(t))^{\ast})$ holds, for any~$k$. Following Lemma~\ref{L:DualTableauxGreenePlactic}, the equality
\[
l_k(R_{row}(t^{\wedge})) = l_k(R_{row}(t))
\] holds for any~$k$, showing by Lemma~\ref{L:GreeneInvariants} that the super tableaux~$t$ and~$t^{\wedge}$ have the same shape.
\end{proof}

\begin{lemma}
\label{L:OppositeTableauLemma2}
Let~$t$ be in~$\YoungT(\Si)$. The super tableaux~$t^{\wedge}$ and~$t^{\mathrm{op}}$ are equal.
\end{lemma}

\begin{proof}
Consider a super tableau~$t$ in~$\YoungT(\Si)$. By construction, the super tableaux~$t$ and~$t^{op}$ have the same shape. Then by Lemma~\ref{L:OppositeTableauLemma} the super tableaux~$t^{\wedge}$ and~$t^{\mathrm{op}}$ are so. We will show by induction on the number of boxes of~$t$ that~$t^{\wedge}=t^{\mathrm{op}}$. This equality holds if~$t$ consists of one box. Suppose that this property is true for super tableaux consisting of~$k-1$ boxes and show it when~$t$ contains~$k$ boxes. Suppose~$t$ has the following form:
\[
\raisebox{0.5cm}{$t\;=\;$}
\scalebox{0.6}{
\begin{tikzpicture}
\draw (0,0) -- ++(0,2) --++(2,0) --++(0,-1) --++(-1,0)  --++(0,-1) --++(-1,0)  -- cycle
node[xshift=0.7cm, yshift=1.4cm]{\scalebox{1.5}{$t'$}}
node[xshift=1.3cm, yshift=2.2cm]{$r$}
node[xshift=0.23cm, yshift=2.2cm,draw]{$x$};
\draw (0.46,2) -- ++(0,0.4) --++(2,0) --++(0,-0.4) --++(-2,0)  -- cycle;
\end{tikzpicture}
}
\]
Note that the super tableau~$t^\ast$ contains one box less than~$t$, and the super tableau~$t^{\mathrm{op}}$ is obtained from~$(t^\ast)^{\mathrm{op}}$ by putting~$x^\ast$ at the same place of the box that was removed from the frame of~$t$. By the induction hypothesis, the super tableaux~$(t^\ast)^{\wedge}$ and~$(t^\ast)^{\mathrm{op}}$ are equal. Since the super tableaux~$t^{\wedge}$ and~$t^{\mathrm{op}}$ have the same shape,  it is sufficient to show that the super tableau~$t^{\wedge}$ can be obtained from~$(t^\ast)^{\wedge}$ by adding a box containing~$x^\ast$ in some place in this super tableau. By construction of~$t^\ast$, the following equivalence~$R_{row}(t^\ast)\sim_{\P(\Si)} R_{row}(t')R_{row}(r)$ holds, showing that~$R_{row}((t^\ast)^{\wedge})\sim_{\P(\Si^{\mathrm{op}})} (R_{row}(r))^\ast(R_{row}(t'))^\ast$. Moreover,  the following equivalence
\[
R_{row}(t^{\wedge})\sim_{\P(\Si^{\mathrm{op}})} (R_{row}(r))^\ast x^\ast(R_{row}(t'))^\ast
\]
 holds. Hence, following the cross-section property~\eqref{Eq:CrossSectionProperty}, when computing the super tableau~$t^{\wedge}$ by insertion we obtain the same super tableau as~$(t^\ast)^{\wedge}$ but with~$x^\ast$ in some box, showing the claim.
\end{proof}

As a consequence of Lemmata~\ref{L:OppositeTableauLemma} and~\ref{L:OppositeTableauLemma2}, we obtain the following result.

\begin{theorem}
\label{T:DualityTableauTheorem}
Let~$t$ be in~$\YoungT(\Si)$. The following properties hold:
\begin{enumerate}[\bf i)]
\item The super tableau~$t^{\mathrm{op}}$ has the same shape as~$t$.\item The words~$(R_{row}(t))^{\ast}$ and $R_{row}(t^{\mathrm{op}})$ are congruent with respect to~$\sim_{\P(\Si^{\mathrm{op}})}$.
\item The map~$t\mapsto  t^{\mathrm{op}}$ is an involution on~$\YoungT(\Si)$.
\end{enumerate}
\end{theorem}

Finally, we deduce the following result.

\begin{corollary}
Let~$t$ be in~$\YoungT(\Si)$ and~$x_1,\ldots,x_k$ be in~$\Si$. The following equality
\begin{equation}
\label{E:DualityRelations}
(x_k \insl{} (  \ldots\insl{}  (x_1\insl{} t)\ldots))
\;
=
\;
\big(((\ldots(t^{\mathrm{op}} \insr{} x_1^{\ast})  \insr{} \ldots )\insr{} x_k^{\ast})\big)^{\mathrm{op}}
\end{equation}
holds in~$\YoungT(\Si)$. 
\end{corollary}

\begin{proof}
Consider~$t$ in~$\YoungT(\Si)$ and~$x_1,\ldots, x_k$ in~$\Si$.
The following equivalence holds:
\[
\begin{array}{rl}
R_{row}(((\ldots(t^{\mathrm{op}} \insr{} x_1^{\ast})  \insr{} \ldots )\insr{} x_k^{\ast}))&\sim_{\P(\Si^{\mathrm{op}})} R_{row}(t^{\mathrm{op}})x_1^{\ast}\ldots x_k^{\ast}\\
&\sim_{\P(\Si^{\mathrm{op}})} (R_{row}(t))^\ast x_1^{\ast}\ldots x_k^{\ast}= \big( x_k\ldots x_1R_{row}(t)\big)^\ast\\
&\sim_{\P(\Si^{\mathrm{op}})} \big(R_{row}( x_k \insl{} ( \ldots\insl{}  (x_1\insl{} t)\ldots))\big)^{\ast}.
\end{array}
\]
Hence, we deduce by Lemma~\ref{L:DualTableauxGreenePlactic} and Theorem~\ref{T:DualityTableauTheorem}  the following equivalence
\[
R_{row}( x_k \insl{} ( \ldots\insl{}  (x_1\insl{} t)\ldots))
\sim_{\P(\Si)} 
R_{row}( \big(((\ldots(t^{\mathrm{op}} \insr{} x_1^{\ast})  \insr{} \ldots )\insr{} x_k^{\ast})\big)^{\mathrm{op}}),
\]
 showing Equality~\eqref{E:DualityRelations} by the cross-section property~\eqref{Eq:CrossSectionProperty}.
\end{proof}

\section{The super Littlewood--Richardson rule}
\label{S:SuperLittlewoodRichardsonRule}

We give in this section an interpretation of the super jeu de taquin in terms of Fomin's  growth diagrams and we deduce remarkable properties of this procedure. Finally, we follow the approach given  in~\cite{Leeuwen01} for the non-signed case in order to  give a combinatorial version of the super Littlewood--Richardson rule.

In the sequel, we will assume that~$\Si$ is finite.

\subsection{Super Schur functions}
\label{SS:SuperSchurFunctions}
Starting from the super plactic monoid~$\P(\Si)$, we can construct a $\mathbb{Z}$-algebra, denoted by~$R_{\Si}$, whose linear generators are the monomials in the monoid~$\P(\Si)$. This is an associative and unitary ring that is not commutative. A generic element in~$R_{\Si}$ can be realized by a formal sum of super  plactic classes with coefficients from~$\mathbb{Z}$. 
Since every super plactic class can be represented by a super tableau, a typical element in~$R_{\Si}$ is a formal sum of super tableaux.
There is a canonical homomorphism from $R_{\Si}$ onto the ring of polynomials~$\mathbb{Z}[X]$ that takes a super tableau $t$ to its monomial~$x^t$, where~$x^t$ is the product of the variables~$x_i$, each occurring as many times in $x^t$ as $i$ occurs in $t$.
Define~$S_\lambda$ (resp.~$S_{\lambda/\mu}$) in~$R_{\Si}$ to be the sum of all super tableaux (resp. super skew tableaux) of shape~$\lambda$ (resp.~$\lambda/\mu$) and entries in~$\Si$, with~$\lambda\in\Pr$ (resp.~$\lambda/\mu$ is a skew shape). 
The image of~$S_{\lambda}$ (resp.~$S_{\lambda}/\mu$) in~$\mathbb{Z}[X]$ gives rise to the so-called \emph{super Schur function} (resp. \emph{super skew Schur function})~$s_{\lambda}(X)$ (resp.~$s_{\lambda/\mu}(X)$).

Let~$\lambda$,~$\mu$ and~$\nu$ be in~$\Pr$ such that~$\lambda/\mu$ is a skew shape. For any~$t$ in~$\YoungT(\Si,\nu)$, we set 
\[
\SkewT(\Si,\lambda/\mu, \dashv t):=\big\{S\in \SkewT(\Si,\lambda/\mu)\;\big|\; \Rec(S)=t\big\},
\]
and we will call the integer
\[
c_{\lambda,\mu}^{\nu} \; := \; \# \SkewT(\Si,\lambda/\mu, \dashv t)
\]
the \emph{super Littlewood--Richardson coefficient}.

\begin{theorem}[The super Littlewood--Richardson rule]
\label{T:SuperLittlewoodRichardsonRule}
Let~$\lambda$,~$\mu$ and~$\nu$ be partitions in~$\Pr$ such that~$\lambda/\mu$ is a skew shape. Fix a super  tableau $t$ of shape~$\nu$. Then the super Littlewood--Richardson coefficient~$c_{\lambda,\mu}^{\nu}$  does not depend on~$t$, and depends only on~$\lambda$,~$\mu$  and~$\nu$. Moreover, the following identity
\begin{equation}
\label{Eq:SuperLittlewoodRichardsonRule}
S_{\lambda/\mu}\;=\; \underset{\nu}{\sum} c_{\lambda,\mu}^{\nu} S_{\nu}
\end{equation}
holds in~$R_{\Si}$.
\end{theorem}

The rest of this section is devoted to prove this result. First, using an interpretation of the super jeu de taquin in terms of Fomin's growth diagrams, Proposition~\ref{P:LittlewoodRichardsonRUleCoefficients} shows that the  coefficient~$c_{\lambda,\mu}^{\nu}$  does not depend on~$t$, and depends only on~$\lambda$,~$\mu$  and~$\nu$. Secondly, if we combine the inverses of bijections~\eqref{Eq:SuperLittlewoodRichardsonBijection} obtained in Proposition~\ref{P:SuperLittlewoodRichardsonBijection} for all partition~$\nu$ in~$\Pr$, we deduce the following bijection
\[
  \SkewT(\Si,\lambda/\mu)\to\underset{\nu\in\Pr}{\bigcup }\SkewT(\Si,\lambda/\mu, \dashv t)\times \YoungT(\Si,\nu) 
\]
for some~$t$ in~$\YoungT(\Si,\nu)$ chosen separately for every~$\nu$, showing that Equality~\eqref{Eq:SuperLittlewoodRichardsonRule} holds in~$R_{\Si}$.

\subsection{Super growth diagrams}
\label{SS:SuperJeuDeTaquinGrowthDiagrams}
The super jeu de taquin procedure can be described using the notion of \emph{growth diagrams} introduced by Fomin,~\cite{Fomin99}. Let $S$  be in~$\SkewT(\Si,\lambda/\mu)$. We represent~$S$ by an \emph{increasing chain of partitions}  in the  super Young lattice~$(\Pr,\subseteq)$ from~$\mu$  to~$\lambda$, that is, a sequence of partitions starting with~$\mu$ and ending with~$\lambda$ such that the Young diagram of each partition is obtained from that of its predecessor by adding exactly one box, so that the boxes are ordered by increasing value of their entries, or in case these entries are equal and belong to distinct columns (resp. the same column), by increasing column (resp. row) number.
A \emph{super growth diagram} is a rectangular array of partitions where every row (resp. column) from left to right (resp. from bottom to top) encodes a super tableau such that all super
tableaux formed by the rows have the same content, as do all super tableaux formed by the
columns. 
After each forward sliding of the super jeu de taquin procedure, we record the corresponding  chain of partitions in a  super growth diagram so that each chain from the topmost row to the bottom-most row represents one of the intermediate super tableaux of this procedure. In particular, the topmost  row represents the initial super skew  tableau~$S$ and the bottom-most one encodes its rectified super  tableau~$\Rec(S)$.
For instance, the super skew tableau~$S$ in Example~\ref{SS:ExampleJeuDeTaquin} is represented by the following chain of partitions:
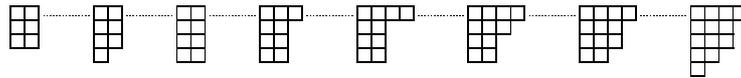
\[
\scalebox{0.6}{
\xymatrix  @C=1em @R=0.1em {
{\ytableausetup{smalltableaux}
\begin{ytableau}
\empty&\empty\\
\empty&\empty\\
\empty&\empty
\end{ytableau}}
	\ar@{.}[rr] ^{}
&&
{\ytableausetup{smalltableaux}
\begin{ytableau}
\empty&\empty\\
\empty&\empty\\
\empty&\empty\\
\empty&\none
\end{ytableau}}
	\ar@{.}[rr] ^{}
&&
{\ytableausetup{smalltableaux}
\begin{ytableau}
\empty&\empty\\
\empty&\empty\\
\empty&\empty\\
\empty&\empty
\end{ytableau}}
	\ar@{.}[rr] ^{}
&&
{\ytableausetup{smalltableaux}
\begin{ytableau}
\empty&\empty&\empty\\
\empty&\empty\\
\empty&\empty\\
\empty&\empty
\end{ytableau}}
	\ar@{.}[rr] ^{}
&&
{\ytableausetup{smalltableaux}
\begin{ytableau}
\empty&\empty&\empty&\empty\\
\empty&\empty\\
\empty&\empty\\
\empty&\empty
\end{ytableau}}
	\ar@{.}[rr] ^{}
&&
{\ytableausetup{smalltableaux}
\begin{ytableau}
\empty&\empty&\empty&\empty\\
\empty&\empty&\empty\\
\empty&\empty\\
\empty&\empty
\end{ytableau}}
\ar@{.}[rr] ^{}
&&
{\ytableausetup{smalltableaux}
\begin{ytableau}
\empty&\empty&\empty&\empty\\
\empty&\empty&\empty\\
\empty&\empty&\empty\\
\empty&\empty
\end{ytableau}}
\ar@{.}[rr] ^{}
&&
{\ytableausetup{smalltableaux}
\begin{ytableau}
\empty&\empty&\empty&\empty\\
\empty&\empty&\empty\\
\empty&\empty&\empty\\
\empty&\empty\\
\empty&\none
\end{ytableau}}
}
}
\]
Its rectification tableau~$\Rec(S)$  is represented by the following chain of partitions:
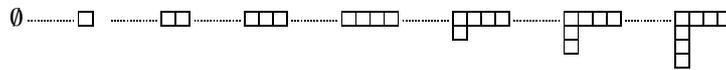
\[
\scalebox{0.6}{
\xymatrix  @C=1em @R=0.1em {
\scalebox{1.5}{{$\emptyset$}}
	\ar@{.}[rr] ^{}
&&
{\ytableausetup{smalltableaux}
\begin{ytableau}
\empty&\none
\end{ytableau}}
	\ar@{.}[rr] ^{}
&&
{\ytableausetup{smalltableaux}
\begin{ytableau}
\empty&\empty
\end{ytableau}}
	\ar@{.}[rr] ^{}
&&
{\ytableausetup{smalltableaux}
\begin{ytableau}
\empty&\empty&\empty
\end{ytableau}}
	\ar@{.}[rr] ^{}
&&
{\ytableausetup{smalltableaux}
\begin{ytableau}
\empty&\empty&\empty&\empty
\end{ytableau}}
	\ar@{.}[rr] ^{}
&&
{\ytableausetup{smalltableaux}
\begin{ytableau}
\empty&\empty&\empty&\empty\\
\empty&\none
\end{ytableau}}
\ar@{.}[rr] ^{}
&&
{\ytableausetup{smalltableaux}
\begin{ytableau}
\empty&\empty&\empty&\empty\\
\empty&\none&\none\\
\empty&\none&\none
\end{ytableau}}
\ar@{.}[rr] ^{}
&&
{\ytableausetup{smalltableaux}
\begin{ytableau}
\empty&\empty&\empty&\empty\\
\empty&\none&\none\\
\empty&\none\\
\empty&\none
\end{ytableau}}
}
}
\]
The super jeu de taquin on~$S$ is then represented by the following super growth diagram:
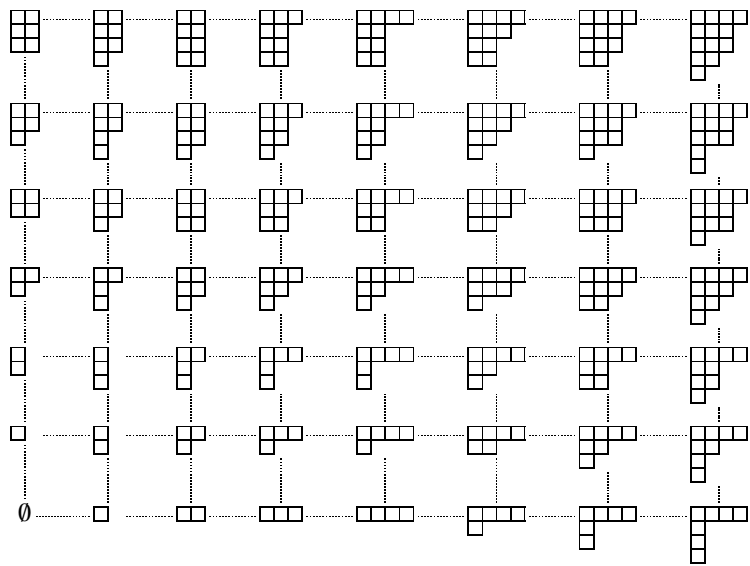
\[
\scalebox{0.6}{
\xymatrix  @C=1em @R=0.1em {
{\ytableausetup{smalltableaux}
\begin{ytableau}
\empty&\empty\\
\empty&\empty\\
\empty&\empty
\end{ytableau}}
	\ar@{.}[rr] ^{}
\ar@{.}[dd] ^{}
&&
{\ytableausetup{smalltableaux}
\begin{ytableau}
\empty&\empty\\
\empty&\empty\\
\empty&\empty\\
\empty&\none
\end{ytableau}}
	\ar@{.}[rr] ^{}
\ar@{.}[dd] ^{}
&&
{\ytableausetup{smalltableaux}
\begin{ytableau}
\empty&\empty\\
\empty&\empty\\
\empty&\empty\\
\empty&\empty
\end{ytableau}}
	\ar@{.}[rr] ^{}
\ar@{.}[dd] ^{}
&&
{\ytableausetup{smalltableaux}
\begin{ytableau}
\empty&\empty&\empty\\
\empty&\empty\\
\empty&\empty\\
\empty&\empty
\end{ytableau}}
	\ar@{.}[rr] ^{}
\ar@{.}[dd] ^{}
&&
{\ytableausetup{smalltableaux}
\begin{ytableau}
\empty&\empty&\empty&\empty\\
\empty&\empty\\
\empty&\empty\\
\empty&\empty
\end{ytableau}}
	\ar@{.}[rr] ^{}
\ar@{.}[dd] ^{}
&&
{\ytableausetup{smalltableaux}
\begin{ytableau}
\empty&\empty&\empty&\empty\\
\empty&\empty&\empty\\
\empty&\empty\\
\empty&\empty
\end{ytableau}}
\ar@{.}[rr] ^{}
\ar@{.}[dd] ^{}
&&
{\ytableausetup{smalltableaux}
\begin{ytableau}
\empty&\empty&\empty&\empty\\
\empty&\empty&\empty\\
\empty&\empty&\empty\\
\empty&\empty
\end{ytableau}}
\ar@{.}[rr] ^{}
\ar@{.}[dd] ^{}
&&
{\ytableausetup{smalltableaux}
\begin{ytableau}
\empty&\empty&\empty&\empty\\
\empty&\empty&\empty\\
\empty&\empty&\empty\\
\empty&\empty\\
\empty&\none
\end{ytableau}}
\ar@{.}[dd] ^{}
\\
\\
{\ytableausetup{smalltableaux}
\begin{ytableau}
\empty&\empty\\
\empty&\empty\\
\empty&\none
\end{ytableau}}
	\ar@{.}[rr] ^{}
\ar@{.}[dd] ^{}
&&
{\ytableausetup{smalltableaux}
\begin{ytableau}
\empty&\empty\\
\empty&\empty\\
\empty&\none\\
\empty&\none
\end{ytableau}}
	\ar@{.}[rr] ^{}
\ar@{.}[dd] ^{}
&&
{\ytableausetup{smalltableaux}
\begin{ytableau}
\empty&\empty\\
\empty&\empty\\
\empty&\empty\\
\empty&\none
\end{ytableau}}
	\ar@{.}[rr] ^{}
\ar@{.}[dd] ^{}
&&
{\ytableausetup{smalltableaux}
\begin{ytableau}
\empty&\empty&\empty\\
\empty&\empty\\
\empty&\empty\\
\empty&\none
\end{ytableau}}
	\ar@{.}[rr] ^{}
\ar@{.}[dd] ^{}
&&
{\ytableausetup{smalltableaux}
\begin{ytableau}
\empty&\empty&\empty&\empty\\
\empty&\empty\\
\empty&\empty\\
\empty&\none
\end{ytableau}}
	\ar@{.}[rr] ^{}
\ar@{.}[dd] ^{}
&&
{\ytableausetup{smalltableaux}
\begin{ytableau}
\empty&\empty&\empty&\empty\\
\empty&\empty&\empty\\
\empty&\empty\\
\empty&\none
\end{ytableau}}
\ar@{.}[rr] ^{}
\ar@{.}[dd] ^{}
&&
{\ytableausetup{smalltableaux}
\begin{ytableau}
\empty&\empty&\empty&\empty\\
\empty&\empty&\empty\\
\empty&\empty&\empty\\
\empty&\none
\end{ytableau}}
\ar@{.}[rr] ^{}
\ar@{.}[dd] ^{}
&&
{\ytableausetup{smalltableaux}
\begin{ytableau}
\empty&\empty&\empty&\empty\\
\empty&\empty&\empty\\
\empty&\empty&\empty\\
\empty&\none\\
\empty&\none
\end{ytableau}}
\ar@{.}[dd] ^{}
\\
\\
{\ytableausetup{smalltableaux}
\begin{ytableau}
\empty&\empty\\
\empty&\empty
\end{ytableau}}
	\ar@{.}[rr] ^{}
\ar@{.}[dd] ^{}
&&
{\ytableausetup{smalltableaux}
\begin{ytableau}
\empty&\empty\\
\empty&\empty\\
\empty&\none
\end{ytableau}}
	\ar@{.}[rr] ^{}
\ar@{.}[dd] ^{}
&&
{\ytableausetup{smalltableaux}
\begin{ytableau}
\empty&\empty\\
\empty&\empty\\
\empty&\empty
\end{ytableau}}
	\ar@{.}[rr] ^{}
\ar@{.}[dd] ^{}
&&
{\ytableausetup{smalltableaux}
\begin{ytableau}
\empty&\empty&\empty\\
\empty&\empty\\
\empty&\empty
\end{ytableau}}
	\ar@{.}[rr] ^{}
\ar@{.}[dd] ^{}
&&
{\ytableausetup{smalltableaux}
\begin{ytableau}
\empty&\empty&\empty&\empty\\
\empty&\empty\\
\empty&\empty
\end{ytableau}}
	\ar@{.}[rr] ^{}
\ar@{.}[dd] ^{}
&&
{\ytableausetup{smalltableaux}
\begin{ytableau}
\empty&\empty&\empty&\empty\\
\empty&\empty&\empty\\
\empty&\empty
\end{ytableau}}
\ar@{.}[rr] ^{}
\ar@{.}[dd] ^{}
&&
{\ytableausetup{smalltableaux}
\begin{ytableau}
\empty&\empty&\empty&\empty\\
\empty&\empty&\empty\\
\empty&\empty&\empty
\end{ytableau}}
\ar@{.}[rr] ^{}
\ar@{.}[dd] ^{}
&&
{\ytableausetup{smalltableaux}
\begin{ytableau}
\empty&\empty&\empty&\empty\\
\empty&\empty&\empty\\
\empty&\empty&\empty\\
\empty&\none
\end{ytableau}}
\ar@{.}[dd] ^{}
\\
\\
{\ytableausetup{smalltableaux}
\begin{ytableau}
\empty&\empty\\
\empty&\none
\end{ytableau}}
	\ar@{.}[rr] ^{}
\ar@{.}[dd] ^{}
&&
{\ytableausetup{smalltableaux}
\begin{ytableau}
\empty&\empty\\
\empty&\none\\
\empty&\none
\end{ytableau}}
	\ar@{.}[rr] ^{}
\ar@{.}[dd] ^{}
&&
{\ytableausetup{smalltableaux}
\begin{ytableau}
\empty&\empty\\
\empty&\empty\\
\empty&\none
\end{ytableau}}
	\ar@{.}[rr] ^{}
\ar@{.}[dd] ^{}
&&
{\ytableausetup{smalltableaux}
\begin{ytableau}
\empty&\empty&\empty\\
\empty&\empty\\
\empty&\none
\end{ytableau}}
	\ar@{.}[rr] ^{}
\ar@{.}[dd] ^{}
&&
{\ytableausetup{smalltableaux}
\begin{ytableau}
\empty&\empty&\empty&\empty\\
\empty&\empty\\
\empty&\none
\end{ytableau}}
	\ar@{.}[rr] ^{}
\ar@{.}[dd] ^{}
&&
{\ytableausetup{smalltableaux}
\begin{ytableau}
\empty&\empty&\empty&\empty\\
\empty&\empty&\empty\\
\empty&\none
\end{ytableau}}
\ar@{.}[rr] ^{}
\ar@{.}[dd] ^{}
&&
{\ytableausetup{smalltableaux}
\begin{ytableau}
\empty&\empty&\empty&\empty\\
\empty&\empty&\empty\\
\empty&\empty&\none
\end{ytableau}}
\ar@{.}[rr] ^{}
\ar@{.}[dd] ^{}
&&
{\ytableausetup{smalltableaux}
\begin{ytableau}
\empty&\empty&\empty&\empty\\
\empty&\empty&\empty\\
\empty&\empty&\none\\
\empty&\none
\end{ytableau}}
\ar@{.}[dd] ^{}
\\
\\
{\ytableausetup{smalltableaux}
\begin{ytableau}
\empty&\none\\
\empty&\none
\end{ytableau}}
	\ar@{.}[rr] ^{}
\ar@{.}[dd] ^{}
&&
{\ytableausetup{smalltableaux}
\begin{ytableau}
\empty&\none\\
\empty&\none\\
\empty&\none
\end{ytableau}}
	\ar@{.}[rr] ^{}
\ar@{.}[dd] ^{}
&&
{\ytableausetup{smalltableaux}
\begin{ytableau}
\empty&\empty\\
\empty&\none\\
\empty&\none
\end{ytableau}}
	\ar@{.}[rr] ^{}
\ar@{.}[dd] ^{}
&&
{\ytableausetup{smalltableaux}
\begin{ytableau}
\empty&\empty&\empty\\
\empty&\none\\
\empty&\none
\end{ytableau}}
	\ar@{.}[rr] ^{}
\ar@{.}[dd] ^{}
&&
{\ytableausetup{smalltableaux}
\begin{ytableau}
\empty&\empty&\empty&\empty\\
\empty&\none\\
\empty&\none
\end{ytableau}}
	\ar@{.}[rr] ^{}
	\ar@{.}[dd] ^{}
&&
{\ytableausetup{smalltableaux}
\begin{ytableau}
\empty&\empty&\empty&\empty\\
\empty&\empty&\none\\
\empty&\none
\end{ytableau}}
\ar@{.}[rr] ^{}
\ar@{.}[dd] ^{}
&&
{\ytableausetup{smalltableaux}
\begin{ytableau}
\empty&\empty&\empty&\empty\\
\empty&\empty&\none\\
\empty&\empty&\none
\end{ytableau}}
\ar@{.}[rr] ^{}
\ar@{.}[dd] ^{}
&&
{\ytableausetup{smalltableaux}
\begin{ytableau}
\empty&\empty&\empty&\empty\\
\empty&\empty&\none\\
\empty&\empty&\none\\
\empty&\none
\end{ytableau}}
\ar@{.}[dd] ^{}
\\
\\
{\ytableausetup{smalltableaux}
\begin{ytableau}
\empty&\none
\end{ytableau}}
	\ar@{.}[rr] ^{}
\ar@{.}[dd] ^{}
&&
{\ytableausetup{smalltableaux}
\begin{ytableau}
\empty&\none\\
\empty&\none
\end{ytableau}}
	\ar@{.}[rr] ^{}
\ar@{.}[dd] ^{}
&&
{\ytableausetup{smalltableaux}
\begin{ytableau}
\empty&\empty\\
\empty&\none
\end{ytableau}}
	\ar@{.}[rr] ^{}
\ar@{.}[dd] ^{}
&&
{\ytableausetup{smalltableaux}
\begin{ytableau}
\empty&\empty&\empty\\
\empty&\none
\end{ytableau}}
	\ar@{.}[rr] ^{}
\ar@{.}[dd] ^{}
&&
{\ytableausetup{smalltableaux}
\begin{ytableau}
\empty&\empty&\empty&\empty\\
\empty&\none
\end{ytableau}}
	\ar@{.}[rr] ^{}
\ar@{.}[dd] ^{}
&&
{\ytableausetup{smalltableaux}
\begin{ytableau}
\empty&\empty&\empty&\empty\\
\empty&\empty&\none
\end{ytableau}}
\ar@{.}[rr] ^{}
\ar@{.}[dd] ^{}
&&
{\ytableausetup{smalltableaux}
\begin{ytableau}
\empty&\empty&\empty&\empty\\
\empty&\empty&\none\\
\empty&\none&\none
\end{ytableau}}
\ar@{.}[rr] ^{}
\ar@{.}[dd] ^{}
&&
{\ytableausetup{smalltableaux}
\begin{ytableau}
\empty&\empty&\empty&\empty\\
\empty&\empty&\none\\
\empty&\none\\
\empty&\none
\end{ytableau}}
\ar@{.}[dd] ^{}
\\
\\
\scalebox{1.5}{{$\emptyset$}}
	\ar@{.}[rr] ^{}
&&
{\ytableausetup{smalltableaux}
\begin{ytableau}
\empty&\none
\end{ytableau}}
	\ar@{.}[rr] ^{}
&&
{\ytableausetup{smalltableaux}
\begin{ytableau}
\empty&\empty
\end{ytableau}}
	\ar@{.}[rr] ^{}
&&
{\ytableausetup{smalltableaux}
\begin{ytableau}
\empty&\empty&\empty
\end{ytableau}}
	\ar@{.}[rr] ^{}
&&
{\ytableausetup{smalltableaux}
\begin{ytableau}
\empty&\empty&\empty&\empty
\end{ytableau}}
	\ar@{.}[rr] ^{}
&&
{\ytableausetup{smalltableaux}
\begin{ytableau}
\empty&\empty&\empty&\empty\\
\empty&\none
\end{ytableau}}
\ar@{.}[rr] ^{}
&&
{\ytableausetup{smalltableaux}
\begin{ytableau}
\empty&\empty&\empty&\empty\\
\empty&\none&\none\\
\empty&\none&\none
\end{ytableau}}
\ar@{.}[rr] ^{}
&&
{\ytableausetup{smalltableaux}
\begin{ytableau}
\empty&\empty&\empty&\empty\\
\empty&\none&\none\\
\empty&\none\\
\empty&\none
\end{ytableau}}
}
}
\]
Note that the leftmost (resp. rightmost) column of a super growth diagram encodes the \emph{recording tableau},  denoted  by~$R$ (resp.~$R'$), filled by elements of~$\Si$, that represents the change in the inner (resp. outer) shape.
Following Theorem~\ref{T:ConfluenceJeuDeTaquin},  the super tableau~$\Rec(S)$ does not depend on the recording tableau~$R$.
For instance, the leftmost column of the above super growth diagram can encode the following tableau:
\[
\raisebox{-0.25cm}{$R\;=\;$}
\scalebox{0.9}{
\ytableausetup{smalltableaux}
\begin{ytableau}
1&3\\
2&4\\
5&6
\end{ytableau}
}
\]
that records the order in which the forward sliding are performed on~$S$. We first choose the inner corner corresponding to the box filled by~$6$, then the inner corner corresponding to the one filled by~$5$, and so on. In this case, the rightmost column encodes the following super skew tableau:
\[
\raisebox{-0.5cm}{$R'\;=\;$}
\scalebox{0.9}{
\ytableausetup{smalltableaux}
\begin{ytableau}
\empty&\empty&\empty&\empty\\
\empty&1&3\\
\empty&2&4\\
\empty&6\\
5&\none
\end{ytableau}}
\raisebox{-0.5cm}{$.$}
\]

Consider a super skew tableau~$S$ in~$\SkewT(\Si,\lambda/\mu)$. We define the following partial bijection
\[
\begin{array}{rl}
\displaystyle\Grw: \underset{\mu\in\Pr}{\bigcup } \YoungT(\Si,\mu)\times\SkewT(\Si,\lambda/\mu)&\displaystyle  \rightharpoonup \underset{\rho\in\Pr}{\bigcup }\YoungT(\Si,\rho)\times\SkewT(\Si,\lambda/\rho)\\
(R,S)&\mapsto (\Rec(S), R')
\end{array}
\]
where~$R$ (resp.~$R'$) denotes the super (resp. skew) tableau corresponding to the leftmost (resp. rightmost) column of the super growth diagram of~$S$.

\subsubsection{Super Fomin's local rule}
\label{SSS:GrowthDiagramLocalRule}
Starting with the topmost row  and the leftmost column,  the entire super growth diagram is completed using the following Fomin's local rule.
Let~$\mu$,~$\rho$,~$\nu$ and~$\lambda$ be partitions in~$\Pr$ that form the following \emph{square of a super growth diagram}:
\[
\scalebox{1}{
\xymatrix  @C=1em @R=0.1em{
{\mu}
\ar@{.}[r] ^{}
\ar@{.}[dd] ^{}
&{\rho}
\ar@{.}[dd] ^{}
\\
\\
{\nu}
\ar@{.}[r] ^{}
&{\lambda}
}
}
\]
that is, both partitions~$\mu$ and~$\lambda$ contain~$\nu$ and~$\rho$ contains both~$\mu$ and~$\lambda$. Then by construction of the super growth diagram, the partition~$\lambda$ is uniquely determined from~$\rho$,~$\mu$ and~$\nu$ by the following \emph{Fomin's local rule}:
\begin{enumerate}[\bf i)]
\item if~$\mu$ is the only partition of its height that contains~$\nu$ and is contained in~$\rho$, then~$\lambda=\mu$,
\item otherwise, there is a unique such partition different from~$\mu$, that is equal to~$\lambda$.
\end{enumerate}

It is worth noting that Fomin's local rule is \emph{symmetric} in~$\lambda$ and~$\mu$, that is, the partition~$\lambda$ is computed from~$\nu$,~$\rho$ and~$\mu$ in the same way as~$\mu$ is computed from~$\nu$,~$\rho$ and~$\lambda$. As a consequence, we deduce the following  symmetry property of the super jeu de taquin:

\begin{lemma}
\label{L:GrowthDiagramSymmetry}
Let~$S$ be in~$\SkewT(\Si)$. If~$\Grw(R,S)=(\Rec(S), R')$ then~$\Grw(\Rec(S), R')=(R,S)$.
\end{lemma}

By this diagrammatic interpretation of the super jeu de taquin, we deduce the following result.
 
\begin{proposition}
 \label{P:LittlewoodRichardsonRUleCoefficients}
Fix partitions $\lambda$, $\mu$, and $\nu$ in~$\Pr$, and suppose that~$t$ is a super tableau of shape~$\nu$. Then the super Littlewood--Richardson coefficient~$c_{\lambda,\mu}^{\nu}$  does not depend on~$t$, and depends only on~$\lambda$,~$\mu$  and~$\nu$.
\end{proposition}

\begin{proof}
Let~$R$ be any super tableau of shape~$\mu$, and consider a super growth diagram with corners~$\emptyset$,~$\mu$,~$\lambda$ and~$\nu$  such that the bottom-most row correspond to~$t$ and the leftmost column encodes~$R$. 
Once we have the four corners, there are a certain number of ways to choose the rightmost column. Once we have that, any choice for the topmost row determines the whole super growth diagram by Fomin's local rule. So the number of ways to complete the super growth diagram, and so the number of topmost rows, is independent of the choice of the bottom-most row corresponding to~$t$,  and depends only on~$\lambda$,~$\mu$  and~$\nu$.
\end{proof}

\subsection{Dual equivalence}
Two super skew tableaux $S$ and $S'$ in~$\YoungT(\Si)$ of the same shape are \emph{dual equivalent} if for any super tableau~$R$ of appropriate shape  such that 
\begin{equation}
\label{Eq:Dualequivalence}
\Grw(R,S)=(\Rec(S), R')\quad \text{ and } \Grw(R,S')=(\Rec(S'), R'')
\end{equation}
we have~$R'=R''$. In other words, the super skew tableaux~$S$ and~$S'$ are dual equivalent if applying the same sequence of  the forward sliding to $S$ and to $S'$ always gives super tableaux of the same shape.

\begin{lemma}
\label{L:Dualequivalence}
Let~$S$ and~$S'$ be in~$\SkewT(\Si,\lambda/\mu)$. If $S$ and $S'$ are both  super jeu de taquin equivalent and dual equivalent, then $S = S'$.
\end{lemma}

\begin{proof}
Since~$S$ and~$S'$ are super  jeu de taquin equivalent, then the super tableaux~$\Rec(S)$ and $\Rec(S')$ are equal by Theorem~\ref{T:ConfluenceJeuDeTaquin}.  Moreover, since~$S$ and~$S'$ are  dual equivalent, then for any super tableau~$R$ of shape~$\mu$ satisfying Property~\eqref{Eq:Dualequivalence}, the super skew tableaux~$R'$ and~$R''$ are equal. Finally,  following Lemma~\ref{L:GrowthDiagramSymmetry}, we have~$\Grw(\Rec(S), R')=(R,S)$  and $\Grw(\Rec(S'), R'')=(R,S')$, showing that~$S=S'$.
\end{proof}

\begin{proposition}
\label{P:SuperLittlewoodRichardsonBijection}
Let~$\mu$, $\nu$ and~$\lambda$ be partitions in~$\Pr$ such that~$\lambda/\mu$ is a skew shape. Fix a super tableau~$t$ of shape~$\nu$. Then there is a bijection 
\begin{equation}
\label{Eq:SuperLittlewoodRichardsonBijection}
\Psi: \SkewT(\Si,\lambda/\mu, \dashv t)\times \YoungT(\Si,\nu) \to \underset{t\in\YoungT(\Si,\nu)}{\bigcup }\SkewT(\Si,\lambda/\mu, \dashv t)
\end{equation}
that can be characterized by the conditions that~$\Psi(S,t')$ and~$S$ are dual equivalent and \linebreak $\Rec(\Psi(S,t'))= t'$, for all~$S$  in~$\SkewT(\Si,\lambda/\mu, \dashv t)$ and $t'$ in~$\YoungT(\Si,\nu)$.
\end{proposition}

\begin{proof}
Consider~$t$ in~$\YoungT(\Si,\nu)$ and let~$S$ be in~$\SkewT(\Si,\lambda/\mu, \dashv t)$ and $t'$  in~$\YoungT(\Si,\nu)$. We will construct~$\Psi(S,t')$ which is uniquely determined by Lemma~\ref{L:Dualequivalence} as follows.
First, choose~$t''$ to be any super tableau in~$\YoungT(\Si,\mu)$, then we have~$\Grw(t'', S) = (t,S')$, for some~$S'$ in~$\SkewT(\Si,\lambda/\nu)$.
Moreover, since~$\Rec(S')=t''$, the following equality~$\Grw(t', S') = (t'', S'')$ holds by Theorem~\ref{T:ConfluenceJeuDeTaquin}, with $S''$ in~$\SkewT(\Si,\lambda/\mu, \dashv t')$. We set~$\Psi(S,t')=S''$. On one hand, we have~$\Rec(\Psi(S,t'))=t'$. On the second hand, since~$\Grw(t'', S) = (t,S')$ and $ \Grw(t'', S'')=(t', S') $ with~$t,t'\in \YoungT(\Si,\nu)$, we deduce  that~$S$ and~$S''$ are dual equivalent. Finally, since the super tableaux~$S$,~$S'$ and~$t'$ can be reconstructed from~$S''$,~$t''$ and~$t$, we deduce that the map~$\Psi$ is bijective.
\end{proof}

\begin{small}
\renewcommand{\refname}{\Large\textsc{References}}
\bibliographystyle{plain}
\bibliography{biblioCURRENT}
\end{small}

\quad

\vfill

\begin{flushright}
\begin{small}
\noindent \textsc{Nohra Hage} \\
\url{nohra.hage@univ-catholille.fr} \\
Faculté de Gestion, Economie \& Sciences (FGES),\\
Université Catholique de Lille,\\
60 bd Vauban, \\
CS 40109, 59016 Lille Cedex, France\\

\end{small}
\end{flushright}

\vspace{0.25cm}

\begin{small}---\;\;\today\;\;-\;\;\hhmm\;\;---\end{small} \hfill
\end{document}